\documentclass[11pt]{amsart}
\usepackage[utf8]{inputenc}
\usepackage[english]{babel}
\usepackage{xcolor}

\usepackage{amssymb}
\usepackage{amsmath}
\usepackage{graphicx}
\usepackage{amsthm}
\usepackage{amscd}
\usepackage[all]{xy}               
\usepackage{tikz-cd}
\usepackage{epsfig,exscale}

 \font \eightrm=cmr8

 \newcommand{\nc}{\newcommand}

\newtheorem{thm}{Theorem}
\newtheorem{exam}{Example}

\newtheorem{cor}[thm]{Corollary}

\newtheorem{lem}[thm]{Lemma}
\newtheorem{prop}[thm]{Proposition}
\newtheorem{defn}{Definition}
\newtheorem{rmk}[thm]{Remark}

\newcommand{\cB}{\mathcal{B}} 
 
\newcommand{\cD}{\mathcal {D}}
\newcommand{\cE}{\mathcal{E}} 
\newcommand{\cF}{\mathcal{F}}

\newcommand{\cI}{\mathcal{I}}

\newcommand{\cL}{\mathcal{L}}

\newcommand{\cP}{\mathcal{P}}

\newcommand{\im}{\mathop{\mathrm{Im}}\nolimits}

\nc{\ignore}[1]{{}}
\nc{\mrm}[1]{{\rm #1}}
\nc{\dirlim}{\displaystyle{\lim_{\longrightarrow}}\,}
\nc{\invlim}{\displaystyle{\lim_{\longleftarrow}}\,}
\nc{\vep}{\varepsilon} \nc{\ep}{\epsilon}
\nc{\sigmat}{\widetilde\sigma}
\nc{\ostar}{\overline{*}}

\nc{\mchar}{\mrm{Char}}
\nc{\Hom}{\mrm{Hom}}
\nc{\id}{\mrm{id}}

\nc{\remark}{\noindent{\bf{Remark:}}}
\nc{\remarks}{\noindent{\bf{Remarks:}}}

 \nc{\delete}[1]{}
 \nc{\grad}[1]{^{({#1})}}
 \nc{\fil}[1]{_{#1}}

\nc{\BA}{{\mathbb A}} \nc{\CC}{{\mathbb C}} \nc{\DD}{{\mathbb D}}
\nc{\EE}{{\mathbb E}} \nc{\FF}{{\mathbb F}} \nc{\GG}{{\mathbb G}}
\nc{\HH}{{\mathbb H}} \nc{\LL}{{\mathbb L}} \nc{\NN}{{\mathbb N}}
\nc{\PP}{{\mathbb P}} \nc{\QQ}{{\mathbb Q}} \nc{\RR}{{\mathbb R}}
\nc{\TT}{{\mathbb T}} \nc{\VV}{{\mathbb V}} \nc{\ZZ}{{\mathbb Z}}
\nc{\setr}{\RR} \nc{\setn}{\NN}
\nc{\Cal}[1]{{\mathcal {#1}}}
\nc{\mop}[1]{\mathop{\hbox {\rm #1} }}
\nc{\smop}[1]{\mathop{\hbox {\eightrm #1} }}
\nc{\mopl}[1]{\mathop{\hbox {\rm #1} }\limits}
\nc{\frakg}{{\mathfrak g}}
\nc{\g}[1]{{\mathfrak {#1}}}
\def \restr#1{\mathstrut_{\textstyle |}\raise-8pt\hbox{$\scriptstyle #1$}}
\def \srestr#1{\mathstrut_{\scriptstyle |}\hbox to
  -1.5pt{}\raise-4pt\hbox{$\scriptscriptstyle #1$}}
\nc{\wt}{\widetilde}
\nc{\wh}{\widehat}
\nc{\un}{\hbox{\bf 1}}
\nc{\redtext}[1]{\textcolor{red}{{\tt [[#1]]}}}
\nc{\bluetext}[1]{\textcolor{blue}{#1}}

\nc\fleche[1]{\mathop{\hbox to #1 mm{\rightarrowfill}}\limits}
\def\semi{\!\mathrel{\times}\kern -6.5 pt\joinrel\mathrel{\raise
    1.4pt\hbox{${\scriptscriptstyle |}$}}\,\,}

\nc{\gmop}[1]{\mathop{\hbox {\bf #1} }\nolimits}

\def\fleche#1{\mathop{\hbox to #1 mm{\rightarrowfill}}\limits}
\def\gfleche#1{\mathop{\hbox to #1 mm{\leftarrowfill}}\limits}
\def\inj#1{\mathop{\hbox to #1 mm{$\lhook\joinrel$\rightarrowfill}}\limits}
\def\ginj#1{\mathop{\hbox to #1 mm{\leftarrowfill$\joinrel\rhook$}}\limits}
\def\surj#1{\mathop{\hbox to #1 mm{\rightarrowfill\hskip 2pt\llap{$\rightarrow$}}}\limits}
\def\gsurj#1{\mathop{\hbox to #1 mm{\rlap{$\leftarrow$}\hskip 2pt
      \leftarrowfill}}\limits}
\def \restr#1{\mathstrut_{\textstyle |}\raise-6pt\hbox{$\scriptstyle #1$}}
\def \srestr#1{\mathstrut_{\scriptstyle |}\hbox to
-1.5pt{}\raise-4pt\hbox{$\scriptscriptstyle #1$}}

\nc{\scal}[2]{ \langle #1\, ,\, #2 \rangle}
\nc{\abs}[1]{ \left\vert  #1 \right\vert}
\nc{\norm}[1]{ \left\lVert  #1 \right\rVert}

\newcommand{\Id}{\operatorname{Id}}
\newcommand{\WF}[1]{\mathop{\mathrm{WF}(#1)}\nolimits}

\newcommand{\supp}[1]{{\mathrm{supp}(#1)}}

\newcommand{\Endo}[1]{\mathop{\mathrm{End}(#1)}\nolimits}

\newcommand{\pr}[1]{\mathop{\mathrm{pr}_{#1}}\nolimits}
\newcommand{\fc}{\mathrm{fc}}
\usepackage[babel=true]{csquotes}

\begin{document}

\title[Convolution of distributions on groupoids]
      {About the convolution of distributions on groupoids}

\author{Jean-Marie Lescure , Dominique Manchon, St\'ephane Vassout ${}^{(1)}$} 
\address{ANR-14-CE25-0012-01  SINGSTAR}



\begin{abstract}
We review the properties of transversality of distributions with respect to submersions. This allows us to construct a convolution product for a large class of distributions on Lie groupoids. We get a unital involutive algebra $\cE_{r,s}'(G,\Omega^{1/2})$ enlarging the convolution algebra $C^\infty_c(G,\Omega^{1/2})$ associated with any Lie groupoid $G$.  We prove that $G$-operators  are convolution operators by transversal distributions. We also investigate the microlocal aspects of the convolution product. We give conditions on wave front sets  sufficient to compute the convolution product and we show that the wave front set of the convolution product of two distributions is essentially the product of their wave front sets in the symplectic groupoid $T^*G$ of Coste-Dazord-Weinstein. This also leads to a subalgebra $\cE_{a}'(G,\Omega^{1/2})$ of  $\cE_{r,s}'(G,\Omega^{1/2})$ which contains for instance the algebra of pseudodifferential $G$-operators and a class of Fourier integral $G$-operators which 
will be the central theme of a forthcoming paper. 
\end{abstract}

\maketitle 
\footnotetext[1]{The first and third authors are supported by ANR Grant ANR-14-CE25-0012-01  SINGSTAR}


\section{Introduction}

The motivation of this paper is twofold. Firstly, we wish to study the convolution of distributions on a Lie groupoid and  its relationship with the action of the so-called $G$-operators. Secondly, we would like to set up a neat framework in order to investigate  in a future work the notions of Lagrangian distributions and Fourier integral operators on a groupoid. 

The notion of $C^\infty$ longitudinal family of distributions in the framework of groupoids appears in \cite{MP,NWX,Vassout2006} in order to define right invariant pseudodifferential operators. Also, in the works of  Monthubert \cite{Monthubert}, these families are considered from the point of view  of distributions on the whole groupoid, so that  the action of the corresponding pseudodifferential operators  on $C^\infty$ functions is given by a convolution product. Here, we  carry on this idea by exploring the correspondence between $C^\infty$ longitudinal families of distributions and single distributions on the whole underlying manifold of the groupoid and by studying  the convolution product of distributions on groupoids. This is achieved at two levels. 

The first level is based on the notion of transversality of distributions with respect to a submersion $\pi : M\to B$ \cite{AndrouSkand2011}.  It appears that the space $\cD'_\pi(M)$ of such distributions is isomorphic to the space of $C^\infty$ family of distributions in the fibers of  $\pi$. Also, in the spirit of the Schwartz kernel Theorem suitably stated on the total space of a submersion, the space $\cD'_\pi(M)$ coincides with the space of continuous  $C^\infty(B)$-linear maps between a suitable subspace $C^\infty_{\fc-\pi}(M)$ of $C^\infty$ functions on $M$ and $C^\infty(B)$. Furthermore,  operations such as push-forwards and fibered-products of distributions behave well on transversal distributions and these operations allow to define the convolution product of distributions on groupoids, as soon as these distributions satisfy some transversality assumptions with respect to source or target maps. Distributions on a groupoid which are transversal both to the source and target maps are called bi-
transversal and they give rise to an involutive unital algebra $\cE_{r,s}'(G,\Omega^{1/2})$ for the convolution product. 
Then, one has the necessary tools to prove   that  $G$-operators on 
a groupoid are in $1$ to $1$ correspondence with  transversal distributions acting by convolution and that  bi-transversal distributions are in $1$ to $1$ correspondence with adjointable $G$-operators.

The second level is a microlocal refinement of the first one and consists in using the wave front set of distributions. A basic observation, due to Coste, Dazord and Weinstein \cite{CDW}, is that the cotangent manifold $T^*G$ of any Lie groupoid $G$ carries a non trivial structure of symplectic groupoid over the dual of the Lie algebroid $A^*G$, this structure being intimately related to the multiplication of $G$ and then to the convolution on $C^\infty_c(G,\Omega^{1/2})$. 
This  groupoid combined with the classical calculus of wave front sets developped by  H\"ormander brings in natural conditions on wave front sets of distributions on a groupoid allowing to define their convolution product and to compute the corresponding wave front set using the law of $T^*G$. The main consequence of this approach is that the space of compactly supported {\sl admissible} distributions:
\[
 \cE_{a}'(G,\Omega^{1/2}) = \{ u\in \cE'(G,\Omega^{1/2})\ ;\ \WF{u}\cap \ker s_\Gamma= \WF{u}\cap \ker r_\Gamma=\emptyset \},
\]
where $s_\Gamma,r_\Gamma$ denotes the source and target maps of $T^*G\rightrightarrows A^*G$, is a unital involutive sub-algebra of $(\cE_{r,s}'(G,\Omega^{1/2}),*)$ and that 
\[
  \WF{u*v} \subset \WF{u}* \WF{v},\quad \forall u,v \in \cE_{a}'(G,\Omega^{1/2}),
\]
where $*$ is the multiplication in the Coste-Dazord-Weinstein groupoid $T^*G$. We would like to add that  the corresponding formula of Hormander for the wave front set of composition of kernels \cite{Horm-FIO1,Horm-1} makes the above formula quite predictable. Indeed, given a manifold $X$, the composition of kernels corresponds to convolution in the pair groupoid $X\times X$ and the composition law that H\"ormander defines on $T^*(X\times X)$ to compute wave front sets of composition of kernels is precisely the multiplication map of the Coste-Dazord-Weinstein symplectic groupoid  $T^*(X\times X)$. 

The distributions belonging to $\cE_{a}'(G,\Omega^{1/2})$ are said to have a bi-transversal wave front set. Actually, this second approach of the convolution product of distributions, based on the groupoid $T^*G$ and H\"ormander's techniques, works under assumptions on the wave front sets of distributions weaker than bi-transversality, and we shall briefly develop this point too. However, the algebra $\cE_{a}'(G,\Omega^{1/2})$ is already large enough for the applications that we have in mind. For instance, pseudodifferential $G$-operators are admissible:
\[
 \Psi_c(G) \subset \cE_{a}'(G,\Omega^{1/2}). 
\]
More importantly, if   $\Lambda\subset T^*G\setminus 0$ is  a homogeneous Lagrangian submanifold of $T^*G$ which is also bi-transversal as a subset of $T^*G$, then Lagrangian distributions \cite{Horm-4} subordinated to $\Lambda$ are admissible:
\[
 I^*(G,\Lambda,\Omega^{1/2}) \subset  \cE_{a}'(G,\Omega^{1/2})
\]
and in particular they give rise to $G$-operators. This will be the starting point of a second paper. 
 
The present paper is organized as follows. In section  \ref{sec:2}, we revisit the Schwartz kernel Theorem in the framework of submersions. Then the notion of distributions transversal with respect to a submersion is recalled, we give some examples and we study natural operations available on them. In section \ref{sec:3}, we apply the results of section \ref{sec:2} to the case of groupoids. We then define the convolution product of transversal distributions and obtain the unital algebra $\cE_{r,s}'(G,\Omega^{1/2})$ of bi-transversal distributions. In section  \ref{sec:4}, we link the notion of $G$-operators with the one of transversal distributions and we obtain a $1$ to $1$ correspondence between the space of adjointable compactly supported $G$-operators and $\cE_{r,s}'(G,\Omega^{1/2})$. In section \ref{sec:5}, we use both the H\"ormander's results about wave front sets of distributions and the symplectic groupoid structure on $T^*G$ to identify an important subalgebra of $\cE_{r,s}'(G,\Omega^{1/2})$, 
namely $\cE_{a}'(G,\Omega^{1/2})$ the subspace of distributions with bi-transversal wave front sets, onto which wave front sets behave particulary well with respect to the convolution product. 


Finally, we recall in Section \ref{sec:7} the definition of the Coste-Dazord-Weinstein  groupoid \cite{CDW} and add some explanations and comments.   

The authors would like to mention that the subject of convolution of transversal distributions is also studied in an independent work by E. Van Erp and R. Yuncken \cite{VY}.

{\bf Aknowledgements} \\
We are happy to thank Claire Debord, Georges Skandalis and Robert Yuncken for many enlightening discussions. Also, the present version of our article has greatly benefited from the remarks addressed by two anonymous referees and we would like to warmly thank them.

\section{Distributions, submersions, transversality}\label{sec:2}       
    
\subsection{Schwartz kernel Theorem for submersions}
To handle distributions on groupoids, it is useful to study distributions in the total space of a submersion. The notion of transversality we shall recall is borrowed from \cite{AndrouSkand2011} and it extends the condition of semi-regularity given in \cite[p.532]{Treves1967}.

For any   manifold $M$ and real number $\alpha$, the bundle of $\alpha$-densities is denoted by $\Omega^{\alpha}_M$.  The space $\cD'(M,\Omega^\alpha_M)$ (resp. $\cE'(M,\Omega^\alpha_M)$) is the topological dual of the space $C^\infty_c(M,\Omega^{1-\alpha}_M)$ (resp. $C^\infty(M,\Omega^{1-\alpha}_M)$). With the convention chosen, we have canonical topological embeddings 
\[
  C^\infty(M,\Omega^\alpha) \hookrightarrow \cD'(M,\Omega^\alpha)
\]
and we abbreviate $\cD'(M)=\cD'(M,\Omega^0_M)$, $\Omega_M=\Omega^1_M$.  

Distributions spaces are provided with the strong topology. The space of continuous linear maps between two locally convex vector spaces $E,F$ is denoted by $\cL(E,F)$ and provided with the topology of uniform convergence on bounded subsets. If $E,F$ are modules over an algebra $A$, the subspace of continuous $A$-linear maps between $E$ and $F$ is denoted by $\cL_A(E,F)$ and  considered as a topological subspace of $\cL(E,F)$. 

We are going to reformulate the Schwartz kernel Theorem for distributions in the total space of a submersion $\pi : M\longrightarrow B$   between $C^\infty$-manifolds. To do this, we begin with the product case  $\pi=\pr{1} : X\times Y\longrightarrow X$ where $X\subset \RR^{n_X}$  and $Y\subset \RR^{n_Y}$ denote open subsets.

The Schwartz kernel Theorem then asserts that the map  
\begin{equation}\label{eq:SKT}
      \cD'(X\times Y) \ni u \longmapsto \left(f\longmapsto u_f(x)=\int_Y u(x,y)f(y)dy\right) \in \cL(C^\infty_c(Y),\cD'(X)) 
\end{equation}
where the integral is understood in the distribution sense, is a topological isomorphism. This can be seen as a push-forward operation along the fibers of $\pi$ and to state this for an arbitrary  submersion  $\pi : M\longrightarrow B$ between $C^\infty$-manifolds, we  introduce  the space  
\begin{equation}\label{defn:function-c-pi}
C^\infty_{\fc-\pi}(M) = \{ f\in C^\infty(M)\ ;\ \pi : \supp{f}\to B\text{ is proper } \}. 
\end{equation}
This is  the LF-space associated with the sequence of Frechet spaces 
\(
 \{ f\in C^\infty(M)\ ;\ \supp{f}\subset F_n\}=C^\infty_0(F_n)
\)
where  $(F_n)$  is an exhausting sequence of closed subsets of $M$ such that $\pi : F_n\to B$ is proper. 

\ignore{Following \cite[Theorem 2.1.5]{Horm-1}, we can prove that the topology of $C^\infty_{\fc-\pi}(M)$  is	 given by the semi-norms generated, using a partition of unity, by the following ones expressed in local coordinates:
\begin{equation}
 \label{eq:semi-norm-C-c-pi}
 p_{U,\rho}(f) = \sum_\alpha \sup_{U}\vert\rho_\alpha\partial^\alpha f\vert
\end{equation}
where $U$  and $(\rho_\alpha)$ respectively run  through coordinate patches in $M$ and collections of continuous functions on $M$ such that on any $\Omega$ for which  $\pi : \Omega \to B$ is proper,  
\(
 \# \{\alpha \ ;\ \rho_\alpha\not=0  \text{ on } \Omega \}
\)
is finite and the functions $\rho_\alpha\vert_\Omega$  are bounded  for all $\alpha$. }

The injections $C^\infty_c\hookrightarrow C^\infty_{\fc-\pi} \hookrightarrow C^\infty$ are continuous and $C^\infty_c$ is dense in $C^\infty_{\fc-\pi}$.  When $B$ is compact, we have $C^\infty_{\fc-\pi}=C^\infty_c$. Vector bundles over $M$ can be added and we do not repeat the definitions. 
Then for any $f \in C^\infty_{\fc-\pi}(M,\Omega_M)$, one can associate a distribution $\pi_*(uf)$ on $B$ defined for any  $g \in C^\infty_{c}(B,\Omega_B)$ by
\begin{equation}
\langle \pi_*(uf), g\rangle=\langle uf, g\circ \pi \rangle= \langle u , f.g\circ \pi \rangle.
\end{equation}
One can view naturally $C^\infty_{\fc-\pi}(M,\Omega_M)$ as a $ C^\infty(B)$-module by using $\pi$ :
for $f \in C^\infty_{\fc-\pi}(M,\Omega_M)$ and $g \in C^\infty_{c}(B,\Omega_B)$, one defines $f.g$ on $M$ by $f.g(m)=f(m)g(\pi(m))$ and the condition on the support is obvious.

We  have
\begin{thm}[Schwartz kernel Theorem for submersions]\label{thm:SKT-submersion}
The map 
 \begin{align*}
   \pi_* : \cD'(M) &\longrightarrow \cL_{C^\infty(B)}(C^\infty_{\fc-\pi}(M,\Omega_M), \cD'(B,\Omega_B)) \\
      u & \longmapsto \pi_*(u\cdot)
    \end{align*} 
is a topological isomorphism.
 \end{thm}
 \begin{proof}
 Let $C$ be a bounded subset of $C^\infty_{\fc-\pi}(M,\Omega_M)$ and $D$ be a bounded subset of $C^\infty_c(B)$. Then $C.D=\{f.g\ ; \ f\in C, g\in D\}$ is a bounded subset of $C^\infty_c(M,\Omega_M)$. The continuity of $\pi_*$ follows. Conversely, we define
 $I :\cL_{C^\infty(B)}(C^\infty_{\fc-\pi}(M,\Omega_M), \cD'(B,\Omega_B))\to \cD'(M)$ by 
 \begin{equation}
  \langle I(T) , f\rangle  = \langle T(f) , \psi \rangle \quad f\in C^\infty_{c}(M,\Omega_M),\  \psi\in C^\infty_c(B), \ f\psi = f.
 \end{equation}
The definition of $I(T)$ as a linear form on $C^\infty_{c}(M,\Omega_M)$ is consistant for $T$ being $C^\infty(B)$-linear, the left hand side does not depend on the choice of $\psi$ such that $f\psi = f$. Moreover, if $E$ is a bounded subset of $C^\infty_{c}(M,\Omega_M)$  then there exists a compact subset $K\subset M$ such that $f\in E$ implies $\supp{f}\subset K$. Fixing $\psi\in C^\infty_c(B)$ such that $\psi=1$ onto $K$ yields that $I(T)$ is a distribution for any $T$ and the continuity of the map $I$. The relations $\pi_*\circ I=\Id$ and $I\circ \pi_*=\Id$ are obvious. 
 \end{proof}
 \begin{rmk}
\rm{Playing with supports, we also get 
\[
   \cE'(M) \simeq \cL_{C^\infty(B)}(C^\infty(M,\Omega_M), \cE'(B,\Omega_B)) \text{ and }
   \cD'_{\fc-\pi}(M) \simeq \cL_{C^\infty(B)}(C^\infty(M,\Omega_M), \cD'(B,\Omega_B)). 
\]
Here $\cD'_{\fc-\pi}(M)$  is the topological dual of the LF-space $\{ f\in C^\infty(M,\Omega_M)\ ;\  \pi(\supp{f})\text{ is compact}\}$.
}
\end{rmk}
 
\subsection{Transversal distributions}

\begin{defn}\label{defn:transversal-distribution}(\cite{AndrouSkand2011} Androulidakis-Skandalis).
A distribution $u\in\cD'(M)$ is transversal to $\pi$ if $\pi_*(u.f)\in C^\infty(B,\Omega_B)$ for any  $f\in C^\infty_{\fc-\pi}(M,\Omega_M)$.
We note $\cD'_\pi(M)$  the space of $\pi$-transversal distributions. We also set 
\[
\cE'_\pi(M) = \cD'_\pi(M)\cap\cE'(M) \text{ and } \cP'_\pi(M) = \cD'_\pi(M)\cap\cD_{\fc-\pi}'(M) 
\]
\end{defn} 
 Observe that if $u$ is $\pi$-transversal, it follows from the closed graph theorem  for LF-spaces \cite[Cor 1.2.20, p. 22]{BonetPerez} that $\pi_*(u\cdot)\in \cL(C^\infty_{\fc-\pi}(M,\Omega_M),C^\infty(B,\Omega_B))$. This gives
 \begin{prop} Denoting by $\pi_*$ the isomorphism in Theorem \ref{thm:SKT-submersion}, we have
 \begin{equation}\label{eq:transversal-distribution-and-B-linear-maps}
   \pi_*(\cD'_\pi(M)) = \cL_{C^\infty(B)} (C^\infty_{\fc-\pi}(M,\Omega_M),C^\infty(B,\Omega_B)).
 \end{equation}
 \end{prop}
\begin{rmk}
 Similarly,
 \begin{equation}\label{eq:transversal-distribution-and-B-linear-maps-2}
  \pi_*(\cE'_\pi(M)) = \cL_{C^\infty(B)}(C^\infty(M,\Omega_M), C^\infty_c(B,\Omega_B)),
 \end{equation}
 \begin{equation}\label{eq:transversal-distribution-and-B-linear-maps-3} 
  \pi_*(\cP'_\pi(M)) =  \cL_{C^\infty(B)}(C^\infty(M,\Omega_M), C^\infty(B,\Omega_B)).
 \end{equation}
 In all cases, the inverse of the map $\pi_*$ is given by
\begin{equation}
 \langle \pi_*^{-1}(T) ,f\rangle = \int_B T(f) \ , \quad f\in C^\infty_{c}(M,\Omega_M).
\end{equation}	
 When $\pi :X\times Y\to X,\ (x,y)\to x$, the $\pi$-transversal distributions are exaclty the distributions semi-regular with respect to $x$, in the former terminology of  \cite[p.532]{Treves1967}.
\end{rmk}
Actually, transversal distributions are nothing else but $C^\infty$ families of distributions in the fibers of $\pi$. 
In the product case $ \pi : X\times Y\to X, (x,y)\mapsto x$, we are talking about the space $C^\infty(X,\cD'(Y))$
 $C^\infty$ functions on $X$ taking values in the topological vector space $ \cD'(Y)$ \cite{Treves1967}. Since $\cD'(Y)$ is a Montel space, the classical argument using Banach-Steinhaus Theorem shows the useful equivalence
\begin{equation}\label{eq:strong-weak-cvg-seq-product-case}
 u_n\longrightarrow u \text{ in }C^\infty(X,\cD'(Y)) \ \Leftrightarrow \ \forall f\in C^\infty_c(Y),\ \langle u_n,f\rangle\longrightarrow \langle u,f\rangle \text{ in }C^\infty(X).
\end{equation}
This space is generalized as follows for general submersions.
\begin{defn}\label{defn:family-distri-submersion}
A family $u=(u_x)_{x\in B}$  of distributions in the fibers of $\pi$ is $C^\infty$  if 
for any local trivialization   of $\pi$
\[
 U\subset M,\ X\subset B,\ \kappa : U\overset{\simeq}{\longrightarrow} X\times Y,\ \pi|_U=\pi_X\circ \kappa,
\]
we have
\(
 \kappa_*(u|_U)\in C^\infty(X,\cD'(Y)).
\)
The space of  $C^\infty$ families is noted $ C^\infty_\pi(B,\cD'(M))$. The spaces $ C^\infty_{\pi,\mathrm{cpct}}(B,\cE'(M))$ and  $ C^\infty_{\pi}(B,\cD'_{\fc-\pi}(M))$ are defined accordingly. 
\end{defn}
Using a covering of $M$ by local trivializations and a partition of unity, we use the topology of  $C^\infty(X,\cD'(Y))$ to build on $ C^\infty_\pi(B,\cD'(M))$  a complete Hausdorff locally convex vector space structure. 
Concretely, this topology is given by the semi-norms generated by the following ones expressed in local coordinates
\begin{equation}\label{eq:semin-norms-family-distrib}
 p_{k,\cB,K}(u)=\sup_{x\in K,g\in\cB} \sum_{\vert \alpha\vert \le k} \vert \langle \partial^\alpha_x u(x),g(x,\cdot)\rangle\vert
\end{equation}
where $k$ is any integer,  $K$ any compact subset included in a local chart of $B$ and $\cB$  any bounded subset  of $C^\infty_{\fc-\pi}(M,\Omega_M)$. Also, \eqref{eq:strong-weak-cvg-seq-product-case} becomes
\begin{equation}\label{eq:strong-weak-cvg-seq-gen-case}
 u_n\longrightarrow u \text{ in }C^\infty_\pi(B,\cD'(M)) \ \Leftrightarrow \ \forall f\in C^\infty_{\fc-\pi}(M,\Omega_M),\ \langle u_n,f\rangle\longrightarrow \langle u,f\rangle \text{ in }C^\infty(B,\Omega_B).
\end{equation}
 Then 
\begin{prop}\label{prop:family-to-global-distri-submersion-case}
 Using on $\cD'_\pi(M)$ the topology given by \eqref{eq:transversal-distribution-and-B-linear-maps}, the map   
 \begin{align}\label{eq:identification-I}
   C^\infty_\pi(B,\cD'(M))      &  \overset{J}{\longrightarrow}    \cD'_\pi(M)       \\
      u &  \longmapsto  ( f\mapsto  \int_B \langle u_x,f(x,\cdot)\rangle) \nonumber 
 \end{align}
 is a topological isomorphism.
\end{prop}
\begin{proof}
 
 Using the identification $\cD'_\pi(M)\simeq \pi_*(\cD'_\pi(M))$, the map $J$ is given by 
 \[
   J(u)(f)(x) = \langle u_x,f(x,\cdot) \rangle ,\ u\in C^\infty_\pi(B,\cD'(M)), f\in C^\infty_{\fc-\pi}(M,\Omega_M),\ x\in B.
 \]
 Conversely, let us define
 \(
      \pi_*(\cD'_\pi(M))    \overset{E}{\longrightarrow}   C^\infty_\pi(B,\cD'(M))     
 \)
 by 
 \begin{equation}
  \langle E(T)_x ,f\rangle =  T(\widetilde{f})(x)
 \end{equation}
where  $f\in C^\infty_c(\pi^{-1}(x),\Omega_{M}\vert_{\pi^{-1}(x)})$ and $\widetilde{f}\in C^\infty_c(M,\Omega_M)$ is any $C^\infty$ extension of $f$. 
It is easy to check that $E=J^{-1}$ and that the topology given by the semi-norms \eqref{eq:semin-norms-family-distrib} on $C^\infty_\pi(B,\cD'(M))$ coincides with the one given by uniform convergence on bounded subsets for the space $\pi_*(\cD'_\pi(M)) $ through the bijection $J$.
\end{proof}

\begin{rmk}\label{rmk:some-inversion-formula}
 
 We similarly get 
 \[
  C^\infty_{\pi,\mathrm{cpct}}(B,\cE'(M))\simeq \cE'_\pi(M) \text{ and } C^\infty_{\pi}(B,\cD'_{\fc-\pi}(M))  \simeq \cP'_\pi(M)  .
 \]
 If vector bundles $E$ over $M$ and  $F$ over $B$ are given, we obtain canonical embeddings
\begin{equation}\label{eq:P-with-bundle-over-base}
 \cD'_\pi(M,E)\hookrightarrow\cD'_\pi(M,E\otimes\Endo{\pi^*F}) \simeq \cL_{C^\infty(B)}(C^\infty_{\fc-\pi}(M,\Omega_M\otimes E^*\otimes\pi^*F), C^\infty(B,\Omega_B\otimes F))
\end{equation}
and
\begin{equation}\label{eq:I-with-bundle-over-base}
   C^\infty_\pi(B,\cD'(M,E)\hookrightarrow C^\infty_\pi(B,\cD'(M,E\otimes\Endo{\pi^*F}))  \simeq   \cD'_\pi(M,E\otimes\Endo{\pi^*F}) .
 \end{equation}
\end{rmk}

\subsection{Examples of transversal distributions}\ 

Obviously, if $\pi : M\to M$ is the identity map then $\cD'_\pi(M)=C^\infty(M)$ and if $\pi$ maps $M$ to a point then $\cD'_\pi(M)=\cD'(M)$. \\
The wave front set (\cite[Chapter 8]{Horm-1}) is a powerful tool to analyse the singularities of a distribution. It can be thought of as the set of directed points in $T^*M\setminus 0$, around which the Fourier transform is not rapidly decreasing. Using wave front set
is a convenient way to check the transversality of distributions with respect to a given submersion $\pi : M\to B$, and it thus gives access to more interesting examples. Indeed,   
\begin{prop}\label{prop:global-distri-to-family-submersion-case}
  Let $W\subset T^*M\setminus 0$ be a closed cone and $\cD'_W(M)=\{ u\in\cD'(M)\ ;\ \WF{u}\subset W\}$. If   $W\cap (\ker d\pi)^\perp=\emptyset$, then 
\[
 \cD'_W(M) \subset \cD'_\pi(M).
\]
\end{prop}
\begin{proof}
 We apply the formula (3.6) p. 328 of \cite{GuilleStenb1977}: 
\[
 \WF{\pi_*(u.f)}\subset (d\pi)_*(\WF{u.f})\subset (d\pi)_*(\WF{u})=\{(x,\xi)\ ;\ x=\pi(m)\ , (m,{}^td\pi_m(\xi))\in \WF{u}\}.
\]
Since  $(\ker d\pi)^\perp= \{(m,\zeta)\ ;\   \zeta\in\im({}^td\pi_m) \}$, we obtain $\WF{\pi_*(u.f)}=\emptyset$, and thus $\pi_*(u.f)$ is smooth. 
\end{proof}
For instance, consider a section of $\pi$, that is a submanifold $X\subset M$ such that $\pi : X\to B$ is a diffeomorphism onto an open subset of $B$. Let $\omega\in \Omega(X)$ be any $C^\infty$ density and define $l_\omega\in\cD'(M,\Omega_M)$ by 
\begin{equation}\label{eq:densities-on-sections}
 \langle l_\omega, f \rangle =\int_X f\omega .
\end{equation}
Then $l_\omega\in\cD'_\pi(M,\Omega_M)$, for $\WF{l_\omega}\subset N^*(X)$  (see \cite[Example 8.2.5]{Horm-1}) and $N^*(X)\cap (\ker d\pi)^\perp =X\times \{0\}$. Alternatively, it is easy to check that $\pi_*(l_\omega.f)$ is given by the $C^\infty$ density $\pi_*(\omega f\vert_X)$.  Of course, for any differential operator $P$ on $M$, we still have $Pl_\omega\in \cD'_\pi(M,\Omega_M)$, for $\WF{Pu}\subset \WF{u}$ for any distribution $u$. Actually, this gives all instances of transversal distributions supported within a section. Indeed, let $u\in \cE'_\pi(M,\Omega_M)$ such that $\supp{u}\subset X$. It is no restriction to work in a local trivialization, that is to assume $\pi :M=X\times \RR^n\to X,\ (x,y)\mapsto x$ and identify $X\simeq X\times\{0\}$. By \cite[Theorem 2.3.5]{Horm-1}, we have
\begin{equation}
 \langle u,\phi \rangle = \sum_{\vert\alpha\vert\le k} \langle u_\alpha,(\partial^\alpha_y\phi)(\cdot,0) \rangle, \qquad \forall \phi\in C^\infty_c(X\times \RR^n)
\end{equation}
where $k$ is the order of $u$ and  $u_\alpha\in \cD'(X)$ has order $k-\vert\alpha\vert$. It follows that 
\begin{equation}
 C^\infty(X)\ni \pi_*(fu) = \sum_{\vert\alpha\vert\le k}  (\partial^\alpha_y f)(\cdot,0).u_\alpha,\quad \forall f\in C^\infty(X\times \RR^n).
\end{equation}
Selecting $f=y^\alpha$  shows that  $u_\alpha$ is $C^\infty$.
We have proved
\begin{prop}\label{prop:transversal-layers}
 Let $u\in \cE'(M,\Omega_M)$ such that $\supp{u}\subset X$, $X$ being a section of $\pi$. Then $u\in\cE'_\pi(M,\Omega_M)$ if and only if $u$ is a finite sum of distributions obtained by differentiation along the fibers of $\pi$ of distributions of the kind \eqref{eq:densities-on-sections}.
\end{prop}

\begin{rmk}\label{rmk:transv-subm-not-imply-transv-wf}
 $u\in\cD'_\pi(M)$ does not imply $\WF{u}\cap\ker d\pi^\perp=\emptyset$. Indeed, consider $\pi : \RR\times\RR\to\RR, (x,y)\mapsto x$ and define $u\in C^\infty(\RR,\cD'(\RR))$ by 
\begin{equation}
 \langle u,\phi \rangle(x) = \sqrt{2\pi}\int \chi(\eta)\vert\eta\vert e^{-\eta^2x^2/2} \hat{\phi}(-\eta)d\eta
\end{equation}
where $\chi$ is $C^\infty$, $\chi(\eta)=1$ if $\vert\eta\vert\ge 1$ and  $\chi(\eta)=0$ if $\vert\eta\vert\le 1/2$. Since 
\(
 \hat{u}(\xi,\eta) = \chi(\eta)  e^{-\xi^2/(2\eta^2)}
\) we conclude $\WF{u}\cap(\ker d\pi)^\perp\not=\emptyset$ (\cite[Section 8.1]{Horm-1}).
\end{rmk}

It is not obvious to us how to characterize transversal distributions whose wave front set avoids $(\ker d\pi)^\perp$. We give in the following lemma a sufficient condition.  
\begin{lem}\label{lem:some-families-to-global-distri}
Let $v \in \cD'_{\pi_X}(X\times Y)$ and assume that there exists constants $d\in\NN$ and $\delta\in[0,1)$ such that for any compact subset $K$ of $Y$ and multi-index $\beta\in\NN^{n_X}$, one can find a constant $C_{K\beta}$ such that 
\begin{equation}\label{eq:fam-to-global-assumption}
 |\langle \partial^\beta v_x, f \rangle| \le C_{K\beta}\| f\|_{K,d+\delta|\beta|},\quad \forall f\in C^\infty_c(Y), \ x\in X.
\end{equation}
Here, we have set $\| f\|_{K,d+\delta|\beta|} = \sum_{|\alpha|\le d+\delta|\beta|}\sup_K|\partial^\alpha f|$.
Then we have $\WF{v}\subset (\ker d\pi_Y)^\perp$.\\
In particular,  $\WF{v}\cap (\ker d\pi_X)^\perp =\emptyset$.
\end{lem}
\begin{rmk}
 \rm{Distributions in Proposition \ref{prop:transversal-layers} satisfy the assumption of the lemma with $\delta=0$.}
\end{rmk}

\begin{proof}[Proof of the lemma]
Let us fix $(x_0,y_0,\xi_0,\eta_0)\not\in (\ker d\pi_Y)^\perp$, that is, $\xi_0\not=0$ and assume that $|(\xi_0,\eta_0)|=1$. We work below  in a conic neighborhood $\Gamma$ of $(x_0,y_0,\xi_0,\eta_0)$ such that for all $(x ,y,\xi,\eta)\in\Gamma$ with $|(\xi,\eta)|=1$, we have $|\xi_j|\ge |\xi_{0j}|/2$ for some fixed $j$ such that $\xi_{0j}\not=0$.

Let $(x ,y,\xi,\eta)\in\Gamma$ be such that $|(\xi,\eta)|=1$ and $\varphi(x,y)$ be supported in a compact neighborhood 
$K\times L$ of $(x_0,y_0)$ in $X\times Y$. Denoting $\varphi_x=\varphi(x,\cdot)$, we have for any $N>0$
\begin{eqnarray}\label{eqn:WF-computation-1}
|\langle v,\varphi e^{-it\langle (\cdot,\cdot),(\xi,\eta)\rangle}\rangle| &=& |\int \scal{v_x}{\varphi_xe^{-it\scal{(x,\cdot)}{(\xi,\eta)}}}dx| 
 = |\int \widehat{\varphi_xv_x}(t\eta)e^{-it\scal{x}{ \xi}}dx| \nonumber  \\
 &\le& C.\left(\sum_{|\alpha|\le N} \sup_{x\in L}|\partial^\alpha_x\widehat{\varphi_xv_x}(t\eta) ||\xi|^{|\alpha|-2N}\right)t^{-N}  \text{ by \cite[Theorem 7.7.1]{Horm-1}}.
\end{eqnarray}
Moreover, since $v :x\mapsto v_x$ is $C^\infty$, we have 
\[ 
\partial^N_{x_j} \widehat{\varphi_xv_x}(t\eta) = \partial^N_{x_j} \langle \varphi_xv_x,  e^{-it\scal{\cdot}{ \eta}}\rangle
=\langle \partial^N_{x_j}\varphi_xv_x,  e^{-it\scal{\cdot}{ \eta}}\rangle = \widehat{\partial^N_{x_j}\varphi_xv_x}(t\eta).
\]

We note $K_\epsilon=\{y+z;\ y\in K,|z|<\epsilon\}$ for any $\epsilon>0$ and let $\chi_\epsilon\in C^\infty_c(K_\epsilon)$  be such that $\chi_\epsilon=1$ on $K_{\epsilon/2}$.  If $H(\eta)$ denotes the supporting function of $K$ \cite[4.3.1]{Horm-1}, we get using the assumption \eqref{eq:fam-to-global-assumption} and the  proof of the Paley-Wiener-Schwartz Theorem in \cite[7.3.1]{Horm-1}
\begin{eqnarray*}
 |\widehat{\partial^N_{x_j}\varphi_xv_x}(\eta)| 
       &=& |\partial^N_{x_j}\varphi_xv_x(\chi_\epsilon e^{-i \scal{\cdot}{ \eta}})|
     \le  C_{K_{\epsilon}N}\sum_{|\beta|\le d+\delta N}\sup|\partial^\beta(\chi_\epsilon e^{-i \scal{\cdot}{ \eta}})|\\
 &\le & C_{K_{\epsilon}N}.C.e^{H(0)}. \sum_{|\beta|\le d+\delta N} \epsilon^{-\beta}(1+|\eta|)^{d+\delta N-|\beta|}.
\end{eqnarray*}
With $\epsilon=1/(1+|\eta|)$ and using the inequalities $C_{K_{\epsilon}N}\le C_{K_{\epsilon'}N}$ if $\epsilon<\epsilon'$, we obtain
 \begin{equation}\label{eq:some-families-to-global-distri-3}
 |\widehat{\partial^N_{x_j}\varphi_xv_x}(\eta)| 
 \le  C_{K_{1}N}.C.(1+|\eta|)^{d+\delta N} 
 \le  C'_{KN}(1+|\eta|)^{d+\delta N}.
\end{equation}
Using uniform estimates $|\xi|\ge c_1>0$ and  $(1+|t\eta|)\le c_2t$ for  $(\xi,\eta)\in\Gamma,\ |(\xi,\eta)|=1$ and the estimate \eqref{eq:some-families-to-global-distri-3} applied to \eqref{eqn:WF-computation-1}, we get
\begin{eqnarray*}
 |\langle v,\varphi e^{-it\langle(\cdot,\cdot),(\xi,\eta)\rangle}\rangle |
   &\le &  C. t^{d+(\delta-1)N}.
\end{eqnarray*}
since $\delta-1<0$, we conclude that $(x_0,y_0,\xi_0,\eta_0)\not\in\WF{v}$. 
\end{proof}

\subsection{Operations on transversal distributions} \ \\
One has obviously, following the definitions.
\begin{prop}\label{prop:restriction-to-sub-submersion}
 Let $E$ be a $C^\infty$ vector bundle over $M$, $C$ a submanifold of $B$ and $\pi_C:\pi^{-1}(C)\to C$ the restriction of the submersion $\pi: M\to B$. The restriction of distributions
\[
\mathrm{Rest}_C : \cD'_{\pi}(M,E)\longrightarrow\cD'_{\pi_C}( \pi^{-1}(C),E\vert_{\pi^{-1}(C)})
\]
is well defined and continuous. 
\end{prop}

\begin{prop}\label{prop:obvious-pull-back}
 Let  $\rho : Z\to M$ and $\pi : M\to B$ be surjective submersions. Let $E$ be a $C^\infty$ vector bundle over $M$. The pull back of distributions restricts to a continuous map
\[
 \rho^*: \cD'_{\pi}(M,E)\longrightarrow\cD'_{\pi\circ\rho }(Z,\rho^*E).
\]
\end{prop}
\begin{proof}
 Since $\rho$ is a submersion, the map 
\[
\begin{matrix}
 \rho_* : & C^\infty_{\fc-\pi\circ\rho}(Z,\Omega_Z\otimes\rho^*E^*)& \longrightarrow & C^\infty_{\fc-\pi}(M,\Omega_M\otimes E^*) \\
         &    f & \longmapsto & (m\mapsto \int_{\rho^{-1}(m)}f)
\end{matrix}
\]
is well defined and  continuous. Since  $\rho^*(u)=u\circ \rho_*$, the proposition follows.  
\end{proof}
 Let $\pi_i : M_i\longrightarrow B$, $i=1,2$ be two submersions and define 
 \[
\pi : M_1\underset{\pi}{\times}M_2=\{(m_1,m_2)\in M_1\times M_2\ ; \ \pi_1(m_1)=\pi_2(m_2)\} \ni (m_1,m_2)\longmapsto \pi_1(m_1).
\]
Writing $\pr{i}(m_1,m_2)=m_i$, we get a commutative square of submersions
\begin{equation}\label{M_fibered_product}
 \xymatrix{
   M_1\underset{\pi}{\times}M_2 \ar^{\pr{2}}[r]\ar^{\pr{1}}[d] & M_2 \ar^{\pi_2}[d] \\
   M_1 \ar^{\pi_1}[r] & B
}
\end{equation}
\begin{prop}\label{prop:comm-square-transv-submersions}
 The pull-back $\mathrm{pr}_1^* : \cD'(M_1) \longrightarrow \cD'(M_1\underset{\pi}{\times}M_2)$ restricts to a continuous map
\begin{equation}\label{eq:pull-back-using-cartesian-square-of-submersions}
 \mathrm{pr}_1^* : \cD'_{\pi_1}(M_1) \longrightarrow \cD'_{\pr{2}}(M_1\underset{\pi}{\times}M_2).
\end{equation}
\end{prop}
\begin{proof}
We identify transversal distributions with $C^\infty$ families and we can work locally, that is we assume that $\pi_j : X\times Y_j\to X$, with $X,Y_1,Y_2$ open subsets in euclidean spaces.  If $u\in C^\infty(X,\cD'(Y_1))$ then $\pr{1}^*(u)$ is given by the family
\[
  X\times Y_2\ni  (x,y_2) \longmapsto u_{x}\in \cD'(Y_1).
\]
The statement follows.   
\end{proof}
\begin{rmk}\rm{
\begin{enumerate}
\item The assertion of  the previous proposition holds for commutative square of surjective submersions
\begin{equation}
 \xymatrix{
   M \ar^{\pr{2}}[r]\ar^{\pr{1}}[d] & M_2 \ar^{\pi_2}[d] \\
   M_1 \ar^{\pi_1}[r] & B
}
\end{equation}
such that any point of $M$, 
\(
 \ker d\pi_1\circ \pr{1}= \ker d\pr{1}+\ker d\pr{2}
\)
or, equivalently, such that $\pr{1} : \pr{2}^{-1}(m_2)\to \pi_1^{-1}(b)$, $b=\pi_2(m_2)$, is a submersion for any $m_2\in M_2$. 
\item One can also prove in the same way that the restriction $\mathrm{pr}_1^* : \cP'_{\pi_1}(M_1) \longrightarrow \cP'_{\pr{2}}(M)$ is continuous, observing that $\mathrm{pr}_1^*$ restricts to 
$
  \cD'_{\fc-\pi_1}(X\times Y_1)\longrightarrow    \cD'_{\fc-\pr{2}}(X\times Y_1\times Y_2).
  $ \\
Then one can, as above, extend this result to more general commutative squares, provided the additional condition that the map $M \to M_1\underset{\pi}{\times} M_2, m\mapsto (\pr{1}(m),\pr{2}(m))$ has compact fibers. 
\end{enumerate}
}\end{rmk}

When a finite set $\cI$ of submersions is given on $M$,
we introduce 
\begin{equation}
 \cD'_{\cI}(M) = \bigcap_{\rho\in\cI} \cD'_{\rho}(M) \subset \cD'(M).
\end{equation}
The space $\cD'_{\cI}(M,E)$ is given the topology generated by the union of the topologies induced by each $\cD'_{\rho}(M) $, $\rho\in\cI$.
We adopt similar convention for the spaces $\cE'_{\cI}(M)$ and $\cD'_{c-\cI}(M)$.
The previous proposition  is now used  to define fibered product of distributions. 
\begin{prop}\label{lem:transv-heredity-fibered-product-case}
We keep the setting of Proposition \ref{prop:comm-square-transv-submersions} and we consider extra submersions $\rho : M_1\longrightarrow A$, $\sigma : M_2\longrightarrow C$.
\begin{center}
\begin{tikzcd}
   M_1\underset{\pi}{\times}M_2 \arrow{r}{\pr{2}}\arrow{d}{\pr{1}}\arrow{dr}{\pi} &   M_2 \arrow{d}{\pi_2}  \arrow{r}{\sigma} &  C \\
   M_1 \arrow{r}{\pi_1} \arrow{d}{\rho} & B  & \\
    A  & &  
\end{tikzcd}
\end{center}
The fibered product of $C^\infty$ functions $(f_1,f_2) \longmapsto f_1\otimes f_2\vert_{M_1\underset{\pi}{\times}M_2}$ 
extends uniquely to separately continuous bilinear maps
\begin{equation}\label{eq:fibered-product-dist-1}
\begin{matrix}
  \cD'_{\pi_1}(M_1)\times \cD'_{\sigma}(M_2) & \longrightarrow & \cD'_{\sigma\circ\pr{2}}(M) \\
   (u_1,u_2)& \longmapsto & u_1\underset{\pi_1}{\times}u_2
\end{matrix}
\hskip 3mm\hbox{ and }\hskip 3mm
\begin{matrix}
  \cD'_{\rho}(M_1)\times \cD'_{\pi_2}(M_2) & \longrightarrow & \cD'_{\rho\circ\pr{1}}(M) \\
   (u_1,u_2)& \longmapsto & u_1\underset{\pi_2}{\times}u_2
\end{matrix}
\end{equation}
 If $u_j\in\cD'_{\pi_j}(M_j)$, $j=1,2$ then  the equality 
\begin{equation}\label{eq:fubini-style-fibered-prod}
 u_1\underset{\pi_1}{\times}u_2 = u_1\underset{\pi_2}{\times}u_2 
\end{equation}
holds  and both previous maps restrict to a separately continuous bilinear map 
\begin{equation}\label{eq:fancy-fibered-product-map}
\begin{matrix}
 \cD'_{\rho,\pi_1}(M_1)\times \cD'_{\pi_2,\sigma}(M_2) &\longrightarrow & \cD'_{\rho\circ\pr{1},\pi,\sigma\circ\pr{2}}(M) \\
 (u_1,u_2) &\longmapsto &  u_1\underset{\pi_1}{\times}u_2 . 
\end{matrix}
\end{equation}
\end{prop}
\begin{rmk}\rm{
The above map in (\ref{eq:fibered-product-dist-1}) restricts to a separately continuous map
\begin{equation}\label{eq:fibered-product-dist-properly-supported}
\cP'_{\pi_1}(M_1)\times \cD'_{\sigma}(M_2)  \longrightarrow  \cD'_{\sigma\circ\pr{2}}(M)\cap  \cD'_{\fc-\pr{2}}(M).
\end{equation}
There are analogous statements for the fibered product over $\pi_2$ and the conditions on supports can be interchanged. 
}\end{rmk}
\begin{proof}
Let $u_1\in\cD'_{\pi_1}(M_1)$ and $u_2\in\cD'(M_2)$. As a  distribution  on $M$, $u_1\underset{\pi_1}{\times}u_2$ is defined by 
\begin{equation}\label{eq:fibered-prod-dist-via-pairing}
 \langle u_1\underset{\pi_1}{\times}u_2, f \rangle = \int_{M_2} \left(\int_{M_{1\pi_2(m_2)}} u_{1\pi_2(m_2)}(m_1)f(m_1,m_2)\right)u_2(m_2)
\end{equation}
where the integrals are taken in the distribution sense. This  coincides with the fibered product of functions when $u_1,u_2$ are $C^\infty$. Proposition \ref{prop:comm-square-transv-submersions} says that 
\begin{equation}\label{eq:proof-transv-heredity-fibered-product-case-2}
 T : u \in  \cD'_{\pi_1}(M_1)\longmapsto \mathrm{pr}_1^*(u)\in  \cL(C^\infty_{\fc-\pr{2}}(M,\Omega_M), C^\infty(M_2,\Omega_{M_2}))
\end{equation}
is well defined and continuous. Observing that the inclusion $C^\infty_{\fc-\sigma\circ\pr{2}}(M,\Omega_M)\subset C^\infty_{\fc-\pr{2}}(M,\Omega_M)$ is continuous, we get the continuity of the mapping:
\begin{equation}\label{eq:proof-transv-heredity-fibered-product-case-2b}
 T : u \in  \cD'_{\pi_1}(M_1)\longmapsto \mathrm{pr}_1^*(u)\in  \cL(C^\infty_{\fc-\sigma\circ\pr{2}}(M,\Omega_M), C^\infty(M_2,\Omega_{M_2})).
\end{equation}
From $\mathrm{pr}_1^*(u)(C^\infty_{\fc-\sigma\circ\pr{2}}(M,\Omega_M))\subset C^\infty_{\fc-\sigma}(M_2,\Omega_{M_2}))$, we get from the closed graph Theorem again that the restricted map 
\[
 \mathrm{pr}_1^*(u) : C^\infty_{\fc-\sigma\circ\pr{2}}(M,\Omega_M)\longrightarrow C^\infty_{\fc-\sigma}(M_2,\Omega_{M_2})
\]
is continuous, and we outline the proof of continuity of the mapping:
\[
 T : u\in  \cD'_{\pi_1}(M_1)\longmapsto \mathrm{pr}_1^*(u)\in  \cL(C^\infty_{\fc-\sigma\circ\pr{2}}(M,\Omega_M), C^\infty_{\fc-\sigma}(M_2,\Omega_{M_2})).
\]
By definition of the topology of uniform convergence on bounded subsets, we have to estimate 
\begin{equation}\label{eq:refined-continuity}
 p_\cB(T(u)) = \sup_{f\in\cB} p(T(u)(f))
\end{equation}
 for any semi-norm  $p$ defining the topology of $C^\infty_{\fc-\sigma}(M_2,\Omega_{M_2})$ and any bounded subset $\cB\subset C^\infty_{\fc-\sigma\circ\pr{2}}(M,\Omega_M)$. But for any such bounded subset $\cB$, there exists $\Omega\subset M$ such that $\sigma\circ\pr{2}: \Omega\to C$ is proper and $f\in \cB\Rightarrow \supp{f}\subset \Omega$. 
We then have $\supp{T(u)(f)} \subset \pr{2}(\Omega)$ for any $\cB$. Since $\sigma : \pr{2}(\Omega)\to C$ is proper, $p$ is also a continuous semi-norm on $C^\infty_0(\pr{2}(\Omega),\Omega_{M_2})$ hence we can replace in \eqref{eq:refined-continuity} $p$ by a semi-norm of $C^\infty_0(\pr{2}(\Omega),\Omega_{M_2})$ or equally by a semi-norm of $C^\infty(M_2,\Omega_{M_2})$ and the result now follows from the continuity of \eqref{eq:proof-transv-heredity-fibered-product-case-2b}.
Then
\begin{equation}\label{eq:fibered-prod-dist-via-linear-map}
  u_1\underset{\pi_1}{\times}u_2= u_2 \circ \mathrm{pr}_1^*(u_1)\in \cL_{C^\infty(C)}(C^\infty_{\fc-\sigma\circ\pr{2}}(M,\Omega_M),C^\infty(C,\Omega_C))
\end{equation}
is continuous in $u_1$ and $u_2$ since the composition of continuous linear maps is separately continuous. 
 
When $u_j\in\cD'_{\pi_j}(M_j)$, $j=1,2$, both fibered products  $u_1\underset{\pi_j}{\times}u_2$, $j=1,2$ makes sense. Starting with \eqref{eq:fibered-prod-dist-via-pairing} and applying Fubini Theorem for distributions, we get their equality and this also allows to take into account the  extra transversality assumptions \eqref{eq:fancy-fibered-product-map} in order to conclude, by the previous method, that  $u_1\underset{\pi_2}{\times}u_2$ is transversal with respect to $\rho\circ\pr{1},\pi$ and $\sigma\circ\pr{2}$ and depends continuously on $u_1$ and $u_2$.
\end{proof}
Consider a commutative diagram 
\begin{equation}
 \xymatrix{   M\ar[rr]^{f} \ar[dr]_{\pi} &  &   N\ar[dl]^{\rho} 
      \\
    &B&}
\end{equation}
where $f$ is a $C^\infty$ map  and $\pi$, $\rho$ are submersions. If $u\in\cE'(M,\Omega_M)$, the push-forward of $u$ by $f$ is given by 
\(
  \langle f_*u, g\rangle = \langle u, g\circ f\rangle
\)
and if moreover $u$ is transversal with respect to $\pi$, then $f_*u$ is given  by the $C^\infty$ family $((f\vert_{M_b})_*u_b),\ b\in B$. 
We obtain a map
\begin{equation}
 f_*:\cE'_\pi(M,\Omega_M)\longrightarrow \cE'_\rho(N,\Omega_N).
\end{equation}
Since $f$ is not necessarily proper, we can not extend $f_*$ to $\cD'_\pi$, nevertheless:
\begin{prop}\label{prop:push-forward-fibered-product}
Let $\varphi\in C^\infty(M)$ such that $f : \supp \varphi\longrightarrow N$ is proper. Then the map
\begin{align}
  \cD'_\pi(M,\Omega_M) & \longrightarrow \cD'_\rho(N,\Omega_N) \\
        u & \longmapsto  f_*(\varphi u)
\end{align}
 is well defined and continuous. 
\end{prop}
\begin{proof}
Under the assumption on the support of $\varphi$, we easily get  that  $g\longmapsto \varphi. g\circ f$ maps continuously $C^\infty_{\fc-\rho}(N)$ into  $C^\infty_{\fc-\pi}(M)$. The result follows. 
\end{proof}

\section{Convolution of transversal distributions on groupoids}\label{sec:3}
We apply these observations in the context of Lie groupoids. \\
A Lie groupoid is a manifold $G$ endowed with the additional following structures:
\begin{itemize}
 \item two surjective submersions $r,s: G\rightrightarrows G^{(0)}$ onto a manifold $G^{(0)}$ called the space of units. 
 \item An embedding $u : G^{(0)}\longrightarrow G$, which allows to consider  $G^{(0)}$ as a submanifold of $G$ and then such that
 \begin{equation}
   r(x) = x \quad ,\quad s(x)=x,\quad \text{ for all } x\in G^{(0)}.
 \end{equation}
 \item A $C^\infty$ map  
 \begin{equation}
   i : G\longrightarrow G, \ \ \gamma\longmapsto \gamma^{-1}
 \end{equation}
 called inversion and satisfying $s(\gamma^{-1})=r(\gamma)$ and $r(\gamma^{-1})=s(\gamma)$ for any $\gamma$.
 \item a $C^\infty$ map 
  \begin{equation}
  m : G^{(2)}=\{(\gamma_1,\gamma_2)\in G^2\ ;\ s(\gamma_1)=r(\gamma_2)\}\longrightarrow G, \ \ (\gamma_1,\gamma_2)\longmapsto \gamma_1\gamma_2
 \end{equation}	
 called the multiplication, satisfying the relations, whenever they make sense 
 \begin{align}
 & (\gamma_1\gamma_2)\gamma_3= \gamma_1(\gamma_2\gamma_3) & & r(\gamma)\gamma=\gamma & & \gamma s(\gamma)=\gamma \\
 &   \gamma\gamma^{-1} = r(\gamma) & & \gamma^{-1} \gamma= s(\gamma)\  & & \ r(\gamma_1\gamma_2)=r(\gamma_1),\ s(\gamma_1\gamma_2)=s(\gamma_2) .
 \end{align}
\end{itemize}
It follows from these axioms that $i$ is a diffeomorphism equal to its inverse, $m$ is a surjective submersion and  $\gamma^{-1}$ is the unique inverse of $\gamma$, for any $\gamma$, that is the only element of $G$ satisfying $ \gamma\gamma^{-1} = r(\gamma), \ \gamma^{-1} \gamma= s(\gamma)$. These assertions need a proof, and the unfamiliar reader is invited to consult for instance \cite{Mackenzie2005} and references therein. 

It is customary to write
\[
 G_x=s^{-1}(x),\quad G^x=r^{-1}(x),\ G_x^y= G_x\cap G^y, \ m_x=m\vert_{G^x\times G_x} : G^x\times G_x\longrightarrow G.
\]
$G_x$, $G^x$ are submanifolds and $G_x^x$ is a Lie group. The submersion $d :(\gamma_1,\gamma_2)\mapsto \gamma_1\gamma_2^{-1}$ defined on $G\underset{s}\times G$ is called division of $G$. 

Obviously, Lie groups, $C^\infty$ vector bundles, principal bundles, are Lie groupoids. Also, for any manifold $X$, the manifold $X\times X$ inherits a canonical structure of Lie groupoid with unit space $X$ and multiplication given by $(x,y).(y,z)=(x,z)$. The reader can find in \cite{Winkelnkemper1983,Pradines1986,Connes1994,NWX,MP,Debord2001,Monthubert1999,Nistor2000,DLR} more concrete examples.

The Lie algebroid $A(G)$ of a Lie groupoid $G$ is the fiber bundle $TG\vert_{G^{(0)}}/TG^{(0)}$ over $G^{(0)}$. It can be identified with $\hbox{Ker }ds\vert_{G^{(0)}}$ or $\hbox{Ker }dr\vert_{G^{(0)}}$. Its dual $A^*(G)$ is the conormal bundle of $G^{(0)}$.

We  recall the construction of the canonical  convolution algebra $C^\infty_c(G,\Omega^{1/2})$ \cite{Connes1994,DebSkand2014} associated with any Lie groupoid $G$. The product of convolution  
\begin{equation}
 C^\infty_c(G,\Omega^{1/2})\times C^\infty_c(G,\Omega^{1/2})\overset{*}{\longrightarrow} C^\infty_c(G,\Omega^{1/2})
\end{equation}
is given by the integral
\begin{equation}\label{eq:basic_formula_convolution_functions}
 f*g(\gamma)=\int_{\gamma_1\gamma_2=\gamma}f(\gamma_1)g(\gamma_2), \quad \gamma\in G
\end{equation}
which is  well defined and gives an internal operation as soon as we take
\begin{equation}
 \Omega^{1/2}=\Omega^{1/2}(\ker dr)\otimes\Omega^{1/2}(\ker ds)=\Omega^{1/2}(\ker dr\oplus \ker ds).
\end{equation}
To understand this point,  we recall
\begin{lem}\label{lem:manip-demi-densite-G}\cite{Connes1994,DebSkand2014}. 
Denoting by $m$ the multiplication map of $G$ and by  $\mathrm{pr}_1,\mathrm{pr}_2 : G\times G\to G$ the natural projection maps, we have a canonical isomorphism 
\begin{equation}\label{eq:why-densities-match-on-gpd}
 \mathrm{pr}_1^*(\Omega^{1/2})\otimes \mathrm{pr}_2^*(\Omega^{1/2})|_{G^{(2)}}\simeq \Omega(\ker dm)\otimes m^*(\Omega^{1/2}).
\end{equation}
\end{lem}
\begin{proof}
We note $G^{(2)}_\gamma$ the fiber of $m$ at $\gamma$, that is 
\(
 G^{(2)}_\gamma = \{ (\gamma_1,\gamma_2)\in G^{(2)}\ ;\ \gamma_1\gamma_2=\gamma\}.
\) 

%

Now, the restricted map $\mathrm{pr}_1:G_\gamma^{(2)}\to G^{r(\gamma)}$ being a diffeomorphism, we have a canonical isomorphism of vector bundles 
\[
 TG^{(2)}_\gamma \overset{\simeq}{\longrightarrow} \mathrm{pr}_1^*(TG^{r(\gamma)})=\mathrm{pr}_1^*(\ker dr)|_{G^{(2)}_\gamma},\ (\gamma_1,\gamma_2,X_1,X_2)\longmapsto  (\gamma_1,\gamma_2,X_1).
\]
Similarly, $TG^{(2)}_\gamma\simeq \mathrm{pr}_2^*(TG_{s(\gamma)})$. Moreover the map
\[
 \mathrm{pr}_1^*(\ker ds)|_{G^{(2)}_\gamma} \overset{\simeq}{\longrightarrow} G^{(2)}_\gamma\times T_\gamma G_{s(\gamma)},\ (\gamma_1,\gamma_2,X_1)\longmapsto  (\gamma_1,\gamma_2, (dR_{\gamma_2})_{\gamma_1}(X_1))
\]
provides a canonical trivialisation of the vector bundle $\mathrm{pr}_1^*(\ker ds)|_{G^{(2)}_\gamma}$. The same holds for 
\[
 \mathrm{pr}_2^*(\ker dr)|_{G^{(2)}_\gamma}\simeq G^{(2)}_\gamma\times T_\gamma G^{r(\gamma)}.
\]
With these isomorphisms in hand, we get 
\begin{eqnarray*}
 \mathrm{pr}_1^*(\Omega^{1/2})|_{G^{(2)}_\gamma}\otimes \mathrm{pr}_2^*(\Omega^{1/2})|_{G^{(2)}_\gamma}&\simeq &\Omega^{1/2}(\mathrm{pr}_1^*(\ker dr \oplus\ker ds)|_{G^{(2)}_\gamma}\oplus\mathrm{pr}_2^*(\ker dr \oplus\ker ds)|_{G^{(2)}_\gamma})\\
&\simeq & \Omega^{1/2} (TG^{(2)}_\gamma\oplus T_\gamma G_{s(\gamma)}\oplus T_\gamma G^{r(\gamma)}\oplus TG^{(2)}_\gamma)\\
&\simeq &\Omega (TG^{(2)}_\gamma)\otimes \Omega^{1/2} ( T_\gamma G_{s(\gamma)}\oplus T_\gamma G^{r(\gamma)}).
\end{eqnarray*}
This gives the canonical isomorphim \eqref{eq:why-densities-match-on-gpd}.
\end{proof}

Since in the basic formula \eqref{eq:basic_formula_convolution_functions}
the function under sign of integration 
\[
 G^{(2)}_\gamma \ni (\gamma_1,\gamma_2)\mapsto f(\gamma_1)g(\gamma_2) \in \left( \mathrm{pr}_1^*(\Omega^{1/2}) \otimes \mathrm{pr}_2^*(\Omega^{1/2})\right)_{(\gamma_1,\gamma_2)}
\]
is  a $C^\infty$ section of the bundle
\(
  \left(\mathrm{pr}_1^*(\Omega^{1/2}) \otimes \mathrm{pr}_2^*(\Omega^{1/2})\right)|_{G^{(2)}_\gamma},
\)
Lemma \ref{lem:manip-demi-densite-G} shows that \eqref{eq:basic_formula_convolution_functions} is the integral of a one density, canonically associated with $f,g$ over the submanifold $m^{-1}(\gamma)$ and that the result is a $C^\infty$ section of $\Omega^{1/2}$. Further computations on densities show that the statement 
\begin{equation}\label{eq:variant_basic_formula_convolution_functions}
 f*g(\gamma)=\int_{G^{r(\gamma)}}f(\gamma_1)g(\gamma_1^{-1}\gamma) = \int_{G_{s(\gamma)}}f(\gamma\gamma_2^{-1})g(\gamma_2)
\end{equation}
makes sense and is true. The involution on $C^\infty_c(G,\Omega^{1/2})$ is also natural in terms of densities 
\[
  f^\star = \overline{i^*(f)},\quad f\in C^\infty_c(G,\Omega^{1/2})
\]
where $i$ is the induced vector bundle isomorphism over the inversion map of $G$
\[
 \ker dr\oplus \ker ds\longrightarrow \ker dr\oplus \ker ds,\ (\gamma,X_1,X_2)\longmapsto (\gamma^{-1},di(X_2),di(X_1)).
\]

The spaces $C^\infty(G,\Omega^{-1/2}\otimes \Omega_G)$ and $C^\infty_c(G,\Omega^{-1/2}\otimes \Omega_G)$ are endowed with their usual Fr\'echet and $\cL\cF$ topological vector space structures and we denote by $\cE'(G,\Omega^{1/2})$ and $\cD'(G,\Omega^{1/2})$ their topological duals. The choice of densities is made so that we have canonical embeddings 
\[
 C^\infty(G,\Omega^{1/2})\hookrightarrow \cD'(G,\Omega^{1/2}) \text{ and } C^\infty_c(G,\Omega^{1/2})\hookrightarrow \cE'(G,\Omega^{1/2}).
\]
For simplicity, we assume in the sequel that $G^{(0)}$ is compact, thus $\cE'(G,\Omega^{1/2})=\cD'_{\fc-\pi}(G,\Omega^{1/2})$ and $\cE'_\pi(G,\Omega^{1/2})=\cP'_{\pi}(G,\Omega^{1/2})$ if $\pi\in\{r,s\}$.
\begin{thm}\label{thm:convolution-transversal-distribution-groupoid}
 The bilinear map  
\begin{align}\label{eq:cont-left-s-tranv-conv}
 \cE'_{s}(G,\Omega^{1/2}) \times \cE'(G,\Omega^{1/2}) & \overset{*}{\longrightarrow} \cE'(G,\Omega^{1/2})\\
     (u,v)& \longmapsto u*v=m_*(u\underset{s}{\times}v)\nonumber 
\end{align}
is well defined and separately continuous. Also, the maps   
\begin{align}\label{eq:cont-left-s-tranv-conv-without-cpct-supp}
 \cD'(G,\Omega^{1/2}) & \overset{*}{\longrightarrow} \cD'(G,\Omega^{1/2}) & \text{ and } && \cD_s'(G,\Omega^{1/2}) & \overset{*}{\longrightarrow} \cD'_s(G,\Omega^{1/2}) \\
     v & \longmapsto u_0*v=m_*(u_0\underset{s}{\times}v) &  && u & \longmapsto u*v_0=m_*(u\underset{s}{\times}v_0)\nonumber 
\end{align}
are well defined and continuous for any $u_0\in \cE'_{s}(G,\Omega^{1/2})$ and $v_0\in \cE'(G,\Omega^{1/2})$. Similar statements are  available for $r$-transversal distributions used as right variables. We get   
by restriction   separately continuous bilinear maps 
\begin{equation}\label{eq:convolution-restricted}
\cE'_{\pi}(G,\Omega^{1/2}) \times \cE'_{\pi}(G,\Omega^{1/2}) \overset{*}{\longrightarrow} \cE'_{\pi}(G,\Omega^{1/2})
\end{equation}
for $\pi=r$ and $\pi=s$. The  space $(\cE'_{\pi}(G,\Omega^{1/2}),*)$ is an associative algebra   with unit given by 
\begin{equation}
 \langle \delta ,f \rangle = \int_{G^{(0)}}f, \quad  f\in C^\infty(G,\Omega^{-1/2}\otimes\Omega_G). 
\end{equation}
In particular $(\cE'_{r,s}(G,\Omega^{1/2}),*)$  is an associative unital algebra with involution given by 
\begin{equation}
  u^\star = \overline{i^*(u)}.
\end{equation}
\end{thm}
\begin{proof}
Applying Proposition \ref{lem:transv-heredity-fibered-product-case}  to the case $M_1=M_2=G$, $B=G^{(0)}$, $\pi_1=s$, $\pi_2=r$ and $\sigma : G\to \{\mathrm{pt}\}$, one gets a distribution $u\underset{s}{\times}v \in \cD'(G^{(2)},\Omega^{1/2})$ which depends continuously on $u$ and $v$.
Since $u\in\cE'$  on can choose $\phi\in C^\infty_c(G)$ such that  $u= \phi u$. Then 
\[
  u \underset{s}{\times}v= \varphi u\underset{s}{\times}v  
\]
where $\varphi =\phi\circ\pr{1}\vert_{G^{(2)}}$ and Proposition \ref{prop:push-forward-fibered-product}  can be applied to the case $f=m$ with $B=\{\hbox{pt}\}$. This gives that  $u*v$ is well defined for $v\in\cD'$ and the continuity of $v\mapsto u*v$ on $\cE',\cD'$ as well.  For fixed $v\in\cE'$, one gets the continuity of $u\mapsto u*v$ on $\cE'_s,\cD'_s$ in the same way. 

To prove the statement involving (\ref{eq:convolution-restricted}) for $\pi=s$ we apply Proposition \ref{lem:transv-heredity-fibered-product-case}  to $M_1=M_2=G$, $B=G^{(0)}$, $\pi_1=s$, $\pi_2=r$ and $\sigma =s$ and Proposition \ref{prop:push-forward-fibered-product} to $\rho=s$ and $\pi =s \circ \pr{2}$.

The associativity of $*$ on distributions follows by continuity and density of $C^\infty_c(G,\Omega^{1/2})$.

We check that  the integral defining $\delta$ has an intrinsic meaning and gives a unit in $\cE'_{r,s}(G,\Omega^{1/2})$.
Since $TG|_{G^{(0)}}=TG^{(0)}\oplus \ker ds$, we have  $\Omega(G)|_{G^{(0)}}=\Omega({G^{(0)}})\otimes \Omega(\ker ds)$.
On the other hand the inversion gives a canonical isomorphism between the bundles $\ker ds|_{G^{(0)}}$ and $\ker dr|_{G^{(0)}}$, thus
\[
 \Omega(\ker ds)=\Omega^{1/2}(\ker ds)\otimes\Omega^{1/2}(\ker ds)\simeq \Omega^{1/2}(\ker ds)\otimes\Omega^{1/2}(\ker dr)=\Omega^{1/2}
\]
Through these canonical identifications, any $f\in C^\infty(G,\Omega^{-1/2}\otimes\Omega_G)$ gives by restriction to $G^{(0)}$ a one density on $G^{(0)}$, which gives a well defined meaning to $\delta(f)$. 
Obviously 
\[
 r_*(f\delta)=s_*(f\delta) = f\vert_{G^{(0)}} \in C^{\infty}(G^{(0)})\subset \cD'(G^{(0)}), \text{ for any } f\in C^\infty(G),
\]
in particular $\delta\in \cE'_{r,s}(G,\Omega^{1/2})$. If $\delta^x\in\cD'(G^x)$, $x\in G^{(0)}$  is the associated $C^\infty$ family, we then get 
by Remark \ref{rmk:some-inversion-formula}  
\[
 \langle \delta^x,\phi \rangle =r_*(\delta \widetilde{\phi})(x) =\phi(x),\ \text{ for any }  \phi\in C^\infty_c(G^x)\text{ and } \widetilde{\phi} \in C^\infty_c(G) \text{ such that  }  \widetilde{\phi}|_{G^x}=\phi .
\]
It follows that for any $ f\in C^\infty(G,\Omega^{-1/2}\otimes\Omega_G)$,
\begin{eqnarray*}
 \langle u*\delta , f \rangle &= & \int_{x\in G^{(0)}} \langle u_x\otimes \delta^x, (f\circ m)|_{G_x\times G^x}\rangle 
 = \int_{x\in G^{(0)}} \langle u_x, f|_{G_x}\rangle 
 = \langle u , f \rangle.
\end{eqnarray*}
The proof of the equality $\delta* u=u$ is similar. The assertion about the involution is obvious. 
\end{proof}
In particular, when one of the two factors is in $C^\infty_c$, the convolution product is defined without any restriction on the other factor. We give a sufficient condition for the result to be $C^\infty$. 
\begin{prop}\label{prop:r-transv-distrib-are-left-multipliers}
The convolution product gives by restriction a bilinear separetely continuous map
\[
 \cD'_{r}(G,\Omega^{1/2})\times C^\infty_c(G,\Omega^{1/2})\overset{*}{\longrightarrow} C^\infty(G,\Omega^{1/2}).
\]
The analogous statement with $C^\infty$ functions on the left and $s$-transversal distributions on the right also holds. The map $u\mapsto u*\cdot$ mapping $\cD'_{r}(G,\Omega^{1/2})$  to  $\cL(C^\infty_c(G,\Omega^{1/2}), C^\infty(G,\Omega^{1/2}))$ is injective.
\end{prop}
\begin{proof}
If $u=(u^y)_y\in\cD'_r$, the map  
\begin{equation} \label{eq:convol-distrib-fonction}
  \gamma\mapsto \langle u^{r(\gamma)}(\cdot), f((\cdot)^{-1}\gamma)\rangle  
\end{equation}
is $C^\infty$ and by definition of the convolution product we get 
\begin{eqnarray*}
\langle u*f , \phi \rangle  =   \int_{\gamma_2\in G} \langle u^{r(\gamma_2)}(\cdot), f((\cdot)^{-1}\gamma_2) \rangle \phi(\gamma_2). 
\end{eqnarray*} 
Thus  $u*f$ coincides with the $C^\infty$ function  \eqref{eq:convol-distrib-fonction}. The continuity of $u\mapsto u*f$ is given by Theorem 
\ref{thm:convolution-transversal-distribution-groupoid} and repeating  the argument given in its proof, one gets the continuity of $f\mapsto u*f$ on  $C^\infty_0(K,\Omega^{1/2})= \{f\in C^\infty\ ;\ \supp{f}\subset K\}$ for any compact $K\subset G$. The results follows by inductive limit. 

Now, the vanishing of $u*f$ for any $f$ and the previous expression for $u*f$ shows that $u^{x}=0$, for any $x$, and thus $u=0$.
  \end{proof}
 
\begin{rmk}\rm{\label{rmk:compact-support}
Note that if in the previous proposition we suppose that $u$ has compact support $K \subset G$, then $u*f$ can be defined for any map $f \in C^\infty(G,\Omega^{1/2})$. Moreover for any $f \in C^\infty_c(G,\Omega^{1/2})$, then $u*f$ is also compactly supported and $\supp{u*f} \subset K .\supp{f}$.
}\end{rmk}
      
\section{G-operators}\label{sec:4}

We recall the notion of $G$-operators given in \cite{MP} and we add a notion of adjoint for them.
\begin{defn}\label{defn:G-op}
 A (left) $G$-operator is a continuous linear map $P : C^\infty_c(G,\Omega^{1/2})\to C^\infty(G,\Omega^{1/2})$ such that there exists a family $P_x : C^\infty_c(G_x,\Omega^{1/2}_{G_x})\longrightarrow C^\infty(G_x,\Omega^{1/2}_{G_x})$, $x\in G^{(0)}$ of operators such that 
\begin{equation}\label{eq:localization-G-op}
  P(f)|_{G_x} = P_x(f|_{G_x}), \ \forall f\in  C^\infty_c(G,\Omega^{1/2}), \ \forall x\in G^{(0)}
\end{equation}
\begin{equation}\label{eq:equiv-G-op}
 P_{r(\gamma)}\circ R_\gamma=R_\gamma \circ P_{s(\gamma)},\ \forall \gamma\in G.
\end{equation}
A $G$-operator $P$ is said adjointable if there exists a  $G$-operator $Q$ such that
\begin{equation}
  (P(f) \vert g ) = (f \vert Q(g) )   \ ; \quad f,g\in C^\infty_c(G,\Omega^{1/2}).
\end{equation}
Here $(f \vert g )= f^\star *g$  is the $C^\infty_c(G,\Omega^{1/2})$-valued pre-hilbertian product of  $C^\infty_c(G,\Omega^{1/2})$.

We note $\mathrm{Op}_G$ and $\mathrm{Op}^\star_{G}$ respectively the linear spaces of $G$-operators and adjointable ones. 

We say that  $G$-operator $P$ is  supported in  $K$ if  $\supp{P(f)}\subset K.\supp{f}$ for all $f$.  The subspaces of compactly supported $G$-operators are denoted $\mathrm{Op}_{G,c}, \mathrm{Op}^\star_{G,c}$. 
\end{defn} 
%
%
Looking at  $C^\infty_c(G,\Omega^{1/2})$ and $ C^\infty(G,\Omega^{1/2})$ as right $ C^\infty_c(G,\Omega^{1/2})$-modules for the convolution product, $G$-operators can be characterized in a simple way. 
\begin{prop}\label{prop:G-op-equivalent-form}
 A linear operator $P : C^\infty_c(G,\Omega^{1/2})\to C^\infty(G,\Omega^{1/2})$ is a $G$-operator if and only if it is continuous and 
\[
  P(f*g)=P(f)*g\quad \forall f,g\in  C^\infty_c(G,\Omega^{1/2}).
\]
In other words, $\mathrm{Op}_G= \cL_{C^\infty_c(G,\Omega^{1/2})}(C^\infty_c(G,\Omega^{1/2}),C^\infty(G,\Omega^{1/2}))$.
\end{prop}
\begin{proof}
Let $P\in \mathrm{Op}_G$. Let us write $p_x$ for the Schwartz kernel of $P_x$. For any $f,g$ compactly supported and $\gamma\in G$
\begin{eqnarray*}
 P(f*g)(\gamma) &=& \int_{\gamma_2\in G_{s(\gamma)}}\int_{\gamma_1\in G_{s(\gamma)}}p_{s(\gamma)}(\gamma,\gamma_2)f(\gamma_2 \gamma_1^{-1})g(\gamma_1)\\
 &=&  \int_{\gamma_1\in G_{s(\gamma)}}\left(\int_{\gamma_2\in G_{s(\gamma)}}p_{s(\gamma)}(\gamma,\gamma_2)(R_{\gamma_1^{-1}}f)(\gamma_2)\right)g(\gamma_1)\\
 &=& \int_{\gamma_1\in G_{s(\gamma)}}\left(\int_{\gamma_2\in G_{s(\gamma)}}p_{r(\gamma_1)}(\gamma\gamma_1^{-1},\gamma_2)f(\gamma_2)\right)g(\gamma_1)\\
&=& \int_{\gamma_1\in G_{s(\gamma)}} P(f)(\gamma\gamma_1^{-1})g(\gamma_1)= P(f)*g(\gamma).
\end{eqnarray*}
Conversely, let $f\in C^\infty_c(G,\Omega^{1/2})$ and $x\in G^{(0)}$ such that $f\vert_{G_x}=0$. observe that  $(g*f)\vert_{G_x}=0$  for any $g\in C^\infty_c(G,\Omega^{1/2})$. It follows that $P(g*f)\vert_{G_x}=P(g)*f\vert_{G_x}=0$. Choose a sequence $\phi_n\in C^\infty_c(G,\Omega^{1/2})$ converging to $\delta$ in $\cE'_r$. Then $\phi_n*f$ converges to $f$ in $C^\infty_c(G,\Omega^{1/2})$ and therefore  
\[
  P(f)(\gamma)= \lim P(\phi_n*f)(\gamma)=0   \quad \forall \gamma\in G_x.
\]
In other words, $P(f)\vert_{G_x}$ only depends on $f\vert_{G_x}$ and we can define  $P_x$ for any $x$ by
\[ 
  P_x(f) = P(\widetilde{f})\vert_{G_x} \quad \forall f\in C^\infty_c (G_x,\Omega^{1/2}_{G_x})\text{ and }\widetilde{f}\in C^\infty_c(G,\Omega^{1/2}) \text{ such that } \widetilde{f}\vert_{G_x} =f.
\]
Let $\gamma \in G^y_x$. Then for any $\gamma'\in G^y$ and $f\in C^\infty_c(G,\Omega^{1/2})$, we have 
\begin{eqnarray*}
 R_\gamma(P_x(\phi_n*f))(\gamma')&=& P(\phi_n*f)(\gamma'\gamma)\\
 &=& P(\phi_n)*f(\gamma'\gamma) = P(\phi_n)*(R_\gamma f)(\gamma')= P(\phi_n*(R_\gamma f))(\gamma').
\end{eqnarray*}
Taking the limit in this equality gives \eqref{eq:equiv-G-op}. 
\end{proof}

Let $u\in \cD'_r(G,\Omega^{1/2})$. Using  Propositions \ref{prop:r-transv-distrib-are-left-multipliers} and \ref{prop:G-op-equivalent-form}, we can define $P\in \mathrm{Op}_G$ by setting $P(f)=u*f$ for any $f\in C^\infty_c(G,\Omega^{1/2})$.

Conversely, let $P\in \mathrm{Op}_G$ and $p_x\in \cD'(G_x\times G_x)$ the Schwartz kernel of $P_x$, $x\in G^{(0)}$. Since 
\[
 \gamma \longmapsto P(f)(\gamma) = \int p_{s(\gamma)}(\gamma,\gamma_1)f(\gamma_1)
\]
is $C^\infty$ for any $f$, we get that $\gamma\mapsto  p_{s(\gamma)}(\gamma,\cdot)$ belongs to $\cD'_{\pr{1}}(G\underset{s}{\times}G)$ and then
using Proposition \ref{prop:restriction-to-sub-submersion}, it restricts to the map $G^{(0)}\ni x\mapsto p_{x}(x,\cdot)$  belonging to $\cD'_s(G)$. Defining $k_P\in\cD'_r(G)$ by $k_P(\gamma)= p_{r(\gamma)}(r(\gamma),\gamma^{-1})$, we get for any $f\in C^\infty_c(G,\Omega^{1/2})$, $x,y\in G^{(0)}$ and $\gamma\in G_x^y$
\begin{eqnarray}
P(f)(\gamma) &=&  \int_{G_x} p_x(\gamma,\gamma_1)f(\gamma_1)
            =  \int_{G_y} p_y(y,\gamma_1)f(\gamma_1\gamma)    \nonumber\\
          &=& \int_{G^y} p_y(y,\gamma_1^{-1})f(\gamma_1^{-1}\gamma)   
            = \langle (k_P)_y, f((\cdot)^{-1}\gamma)\rangle_{G^y} =k_P * f (\gamma).
\end{eqnarray}
Thus $P$ the operator given by left convolution with  $k_P$.  We call $k_P$ {\sl   the convolution distributional kernel} of $P$. Note that $\supp{P}=\supp{k_P}$. 
We have proved 
\begin{thm}\label{thm:G-op-transversal-distrib}
 The map $P\mapsto k_P$ gives  the isomorphisms 
\begin{equation}
 \mathrm{Op}_{G}\simeq\cD'_r(G,\Omega^{1/2})\ \text{ and }\ \mathrm{Op}_{G,c}\simeq\cE'_r(G,\Omega^{1/2}).
\end{equation}
\end{thm} 
If $k_P\in\cD'_{r,s}(G,\Omega^{1/2})$ then $P$ is obviously adjointable and $k_{P^\star}=(k_P)^\star$. Conversely, if $P$ as an adjoint $Q$ then
\begin{equation}
  (k_P*f)^\star * g = (f^\star *k_P^\star )*g = f^\star *(k_Q*g)   \ ; \quad f,g\in C^\infty_c(G,\Omega^{1/2}),
\end{equation}
hence $k_P^\star = k_Q\in \cD'_s(G,\Omega^{1/2})\cap\cD'_r(G,\Omega^{1/2}) $. Thus Theorem \ref{thm:G-op-transversal-distrib} yields
\begin{cor}
 The  map $P\to k_P$ gives an isomorphism
\begin{equation}
  \mathrm{Op}_{G}^\star \simeq  \cD'_{r,s}(G,\Omega^{1/2}).
\end{equation}
\end{cor}
\begin{rmk}
\rm{Rephrazing the previous results, we have, for instance
\[
 \mathrm{Op}_{G} \simeq \cL_{s}(C^\infty_c(G,\Omega^{1/2}),C^\infty(G^{(0)})).
\]
where we have replaced $\cL_{C^\infty(G^{(0)})}$ by $\cL_s$ to emphasize that the $C^\infty(G^{(0)})$-module structure on $C^\infty_c(G,\Omega^{1/2})$ is given by $s$. Also
\[
 \mathrm{Op}^\star_{G} \simeq \cL_{r,s}(C^\infty_c(G,\Omega^{1/2}),C^\infty(G^{(0)})).
\]
where $\cL_{r,s}=\cL_s\cap\cL_r$. 
In terms of Schwartz kernel theorems for submersions, $G$-operators thus appear as \enquote{semi-regular} distributions  (see Treves \cite[p.532]{Treves1967}) since, for $\pi=s$ or $\pi=r$
$$\cD'(G,\Omega^{1/2})\simeq \cL_{\pi}(C^\infty_c(G,\Omega^{1/2}),\cD'(G^{(0)})).$$}
\end{rmk}

Now observe that if $k_P\in \cE_{r,s}'(G,\Omega^{1/2})$, Theorem \ref{thm:convolution-transversal-distribution-groupoid} implies that $P$ extends continuously to a map $\cD'(G,\Omega^{1/2})\longrightarrow \cD'(G,\Omega^{1/2})$ sending the subspace $\cE_{r,s}'$ to $\cE_{r,s}'$. This leads to another characterization of adjointness. 
\begin{prop} 
A compactly supported $G$-operator $P$ is adjointable if and only if it extends continuously to a map
\[
  \widetilde{P} : \cD'(G,\Omega^{1/2})\longrightarrow \cD'(G,\Omega^{1/2})
\]
such that $\widetilde{P}(\delta)\in \cD_{r,s}'(G,\Omega^{1/2})$. 
In that case,  $\widetilde{P}=k_P*\cdot$.
\end{prop}
\begin{proof}
 Let $u\in \cD'(G,\Omega^{1/2})$ and $(u_n)\subset C^\infty_c(G,\Omega^{1/2})$ a sequence converging to $u$ in $\cD'$. We have
\[
 \widetilde{P}(u*f)= \lim P(u_n*f) = \lim P(u_n)*f= \widetilde{P}(u)*f,\quad \forall f\in C^\infty_c(G,\Omega^{1/2}).
\]
Thus $\widetilde{P}$ is automatically $C^\infty_c(G,\Omega^{1/2})$-right linear. It follows that 
\[
 k_P * f = P(f) = P(\delta *f) = \widetilde{P}(\delta) *f ,\quad \forall f\in C^\infty_c(G,\Omega^{1/2})
\]
which proves that $k_P=\widetilde{P}(\delta)\in\cD_{r,s}'(G,\Omega^{1/2})$ and that $\widetilde{P}$ is given by left convolution with $k_P$.
\end{proof}

 \section{Convolution on groupoids and wave front sets  }\label{sec:5}
We now turn to some microlocal aspects of the convolution of distributions on groupoids.  
In view of Proposition \ref{prop:global-distri-to-family-submersion-case}, it is natural to call $r$-transversal any (conic) subset $W\subset T^*G\setminus 0$ such that $W\cap  \ker dr^\perp  = \emptyset$,  indeed in that case 
\begin{equation}
 \cD'_W(G,\Omega^{1/2})\subset\cD'_r(G,\Omega^{1/2}).
\end{equation}
 Similarly, $W$ is called $s$-transversal if $W\cap  \ker ds^\perp  = \emptyset$
and we call bi-transversal any set which is both $r$ and $s$-transversal. We then introduce 
\begin{equation}
 \cD'_{a}(G,\Omega^{1/2}) = \{ u\in \cD'(G,\Omega^{1/2})\ ;\ \WF{u} \text{ is bi-transversal}\}
\end{equation}
and  $\cE'_{a}=\cD'_{a}\cap\cE'$. We call them  {\sl admissible} distributions. From Proposition \ref{prop:global-distri-to-family-submersion-case},  we get
\begin{equation}
 \cD'_{a}(G,\Omega^{1/2})\subset \cD'_{r,s}(G,\Omega^{1/2}).
\end{equation}
\begin{exam}\rm{
Observe that $A^*G\setminus 0$ is  bi-transversal. Since $\Psi(G)= I(G,G^{(0)})\subset \cD'_{A^*G}(G)$  (see \cite{Monthubert})   we get  
\begin{equation}
  \Psi(G)\subset \cD'_{a}(G,\Omega^{1/2}). 
\end{equation}}
\end{exam}
Theorem \ref{thm:convolution-transversal-distribution-groupoid} and Proposition \ref{prop:r-transv-distrib-are-left-multipliers} can be reused in various ways for subspaces of distributions with transversal wave front sets. We only record the main one: the convolution product restricts to a bilinear map
\begin{equation}\label{eq:rough-convolution-adm-distrib}
 \cE'_{a}(G,\Omega^{1/2})\times \cE'_{a}(G,\Omega^{1/2}) \overset{*}{\longrightarrow}\cE'_{r,s}(G,\Omega^{1/2}), 
\end{equation}
and we strenghthen this result as follows, by using the cotangent groupoid structure of Coste-Dazord-Weinstein (see Appendix). 
\begin{thm}\label{thm:E_a-algebra} 
For any $u_1,u_2\in \cE'_{a}(G,\Omega^{1/2})$, we have $u_1*u_2\in \cE'_{a}(G,\Omega^{1/2})$ and 
\begin{equation}\label{eq:The-formula}
 \WF{u_1*u_2} \subset \WF{u_1}*\WF{u_2}
\end{equation}
where on the right, $*$ denotes the product of the symplectic groupoid $T^*G\rightrightarrows A^*G$.
In particular $(\cE'_{a}(G,\Omega^{1/2}),*)$ is a unital involutive subalgebra of $(\cE'_{r,s}(G,\Omega^{1/2}),*)$.
\end{thm}
\begin{proof}
 Let $u_j\in \cE'_{a}(G,\Omega^{1/2})$ and set $W_j=\WF{u_j}$, $j=1,2$. We first show that the fibered product 
$u_1\underset{\pi}{\times}u_2$ (where $\pi=r,s$ indifferently) given by Proposition \ref{lem:transv-heredity-fibered-product-case}, coincides with the distribution obtained by the functorial operations  in \cite[Theorems 8.2.9, 8.2.4]{Horm-1}:
\begin{equation}
 \label{eq:bridge-transversal-functorial-WF}
 u_1\underset{\pi}{\times}u_2 = \rho^*(u_1\otimes u_2) \in\cD'(G^{(2)},\Omega(\ker dm)\otimes m^*(\Omega^{1/2})),
\end{equation}
where $\rho:G^{(2)}\hookrightarrow G^2$. By \cite[Theorem 8.2.9]{Horm-1}), we know that 
\begin{equation}\label{eq:WF-tensor-product-1}
 \WF{u_1\otimes u_2}\subset W_1\times W_2\cup W_1\times (G\times\{0\} )\cup (G\times\{0\})\times W_2,
\end{equation}
and to apply \cite[Theorems 8.2.4]{Horm-1}, we just need to check that 
\begin{equation}\label{eq:WF-restriction-to-G2-1}
 \WF{u_1\otimes u_2}\cap N^*G^{(2)} = \emptyset. 
\end{equation}
Observe that  $N^*G^{(2)}=\ker m_\Gamma\subset \Gamma^{(2)}$ and $\ker ds^\perp=\ker r_\Gamma$. Thus, if 
\[
 \delta_j=(\gamma_j,\xi_j)\in T^*_{\gamma_j}G\text{ and } (\delta_1,\delta_2)\in \WF{u_1\otimes u_2}\cap N^*G^{(2)} 
\]
then $(\delta_1,\delta_2)\in \Gamma^{(2)}$ and 
\begin{equation}
 r_\Gamma(\delta_1)=r_\Gamma(\delta_1\delta_2) = (r(\gamma_1),0). 
\end{equation}
By the $s$-transversality assumption on $W_1$ and the relation \eqref{eq:WF-tensor-product-1}, this implies $\delta_1=(\gamma_1,0)$ and $\delta_2\in W_2$. On the other hand
\begin{equation}
 s_\Gamma(\delta_2)=s_\Gamma(\delta_1\delta_2) = (s(\gamma_2),0),
\end{equation}
which contradicts the $r$-transversality of $W_2$, and this proves \eqref{eq:WF-restriction-to-G2-1}. Therefore, the right hand side in \eqref{eq:bridge-transversal-functorial-WF} is well defined by \cite[Theorems 8.2.4]{Horm-1} and it coincides with the left hand side, which is obvious  after pairing with test functions.  Now 
\begin{equation}\label{eq:convolution-via-Hormander-method}
 u_1*u_2 = m_*(u_1\underset{\pi}{\times}u_2)= m_*\rho^*(u_1\otimes u_2)
\end{equation}
and thus, using   \cite[Theorems 8.2.4]{Horm-1}  and \cite[(3.6), p. 328]{GuilleStenb1977},
\begin{equation}\label{eq:rough-formula-WF-convolution}
 \WF{u_1*u_2} \subset m_*\rho^*\WF{u_1\otimes u_2}.
\end{equation}
Here $\rho^*: T^*G^2 \longrightarrow T^*G^{(2)}$ is the restriction of linear forms and, for any $\widetilde{W}\subset T^*G^{(2)}$,
\[
 m_*(\widetilde{W}) = \{ (\gamma,\xi)\in T^*G \ ;\ \exists (\gamma_1,\gamma_2)\in m^{-1}(\gamma),\ (\gamma_1,\gamma_2,{}^tdm_{\gamma_1,\gamma_2}(\xi))\in \widetilde{W}\cup G^{(2)}\times 0\}.
\]
Since $m$ is submersive, ${}^tdm_{\gamma_1,\gamma_2}$ is injective and the term $G^{(2)}\times 0$ can be removed.  By definition  of the multiplication of $\Gamma=T^*G$, we get, for any $W\subset T^*G^2$, the equivalence  
\begin{equation}
\gamma_1\gamma_2=\gamma \text{ and } (\gamma_1,\gamma_2,{}^tdm_{\gamma_1,\gamma_2}(\xi))\in  \rho^*(W)
\ \Leftrightarrow 
\exists (\delta_1,\delta_2)\in \Gamma^{(2)}\cap W, \  \delta_1\delta_2=(\gamma,\xi).
\end{equation}
Thus, 
\begin{equation}\label{eq:rough-formula-WF-convolution-2}
 m_*\rho^*W =m_\Gamma(W\cap \Gamma^{(2)}).
\end{equation}
By $r$-transversality of $\WF{u_1}$, we have $s_\Gamma(\WF{u_1})\subset A^*G\setminus 0$, so
\(
 \WF{u_1}\times (G\times\{0\}) \cap\Gamma^{(2)}=\emptyset.
\)
Similarly, $s$-transversality of $\WF{u_2}$ gives $(G\times\{0\}) \times  \WF{u_2}  \cap\Gamma^{(2)}=\emptyset$. It follows that 
 $\WF{u_1\otimes u_2}\cap\Gamma^{(2)}=(\WF{u_1}\times \WF{u_2})\cap\Gamma^{(2)}$ and therefore
\[
 m_\Gamma(\WF{u_1\otimes u_2}\cap\Gamma^{(2)})= m_\Gamma((\WF{u_1}\times \WF{u_2})\cap\Gamma^{(2)})=\WF{u_1}* \WF{u_2}
\]
which proves \eqref{eq:The-formula}. Clearly, $W_1*W_2$ is $s$ or $r$-transversal  if the same holds respectively for  $W_1$ and $W_2$, so \eqref{eq:The-formula} implies $u_1*u_2\in\cE'_a$, therefore $\cE'_a$ is a subalgebra of $\cE'_{r,s}$.

Finally, since $\WF{\delta}=A^*G\setminus 0$, we have  $\delta\in\cE'_{a}$ and since $\WF{u^\star}=i_\Gamma(\WF{u})$, we conclude that $\cE'_a$ is unital and involutive. 
\end{proof}
 Looking at the proof of the Theorem, we see that the assumptions on $\WF{u_j}$ can be significanlty relaxed in order to conserve the property  \eqref{eq:WF-restriction-to-G2-1} and then  to be able to define the convolution product $u_1*u_2$ by the right hand side of \eqref{eq:convolution-via-Hormander-method}. 

Firstly, if $W\subset T^*G\setminus 0$, then $W\times (G\times 0) \cap\ker m_\Gamma=\emptyset$. Indeed, if $(\gamma_1,\xi_1,\gamma_2,0)\in W\times(G\times\{0\} )\cap \Gamma^{(2)}$, we can choose $t_1\in T_{\gamma_1} G$ such that $\xi_1(t_1)\not=0$ since $\xi_1\not=0$ by assumption.  Using a local section $\beta$ of $r$ such that $\beta(s(\gamma_1))=\gamma_2$ and setting $t_2=d\beta ds(t_1)\in T_{\gamma_2} G$, we  get 
$(t_1,t_2)\in T_{(\gamma_1,\gamma_2)}G^{(2)}$ and $\xi_1(t_1)+0(t_2)\not=0$, that is $\xi_1\oplus 0\not=0$ which proves that $(\gamma_1,\xi_1,\gamma_2,0)\not\in\ker m_\Gamma$. 

Arguing identically on $(G\times 0)\times  W$  we get the equivalence, for any  distributions $u_1, u_2$ 
\begin{equation}
 \WF{u_1\otimes u_2}\cap \ker m_\Gamma = \emptyset \Leftrightarrow  \WF{u_1}\times\WF{u_2}\cap \ker m_\Gamma = \emptyset .
\end{equation}
This is again the  condition \eqref{eq:WF-restriction-to-G2-1} which is sufficient to define $\rho^*(u_1\otimes u_2)=u_1\otimes u_2\vert_{G^{(2)}}$ and there the convolution product under additional suitable supports conditions. 
\begin{thm}\label{thm:convol-distrib-on-gpd}
 Let $W_j\subset T^*G\setminus 0$ be closed cones such that 
\begin{equation}\label{eq:our-cond-conv-distribution} 
 W_1\times W_2\cap \ker m_\Gamma = \emptyset
\end{equation}
and set $W_1\overline{*} W_2=m_\Gamma((W_1\times W_2 \cup W_1\times 0 \cup 0\times W_2 )\cap\Gamma^{(2)})$. Then the map 
\begin{align}
 \cE_{W_1}'(G,\Omega^{1/2})\times \cE_{W_2}'(G,\Omega^{1/2})& \overset{*}{\longrightarrow}\cE'_{W_1\overline{*} W_2}(G,\Omega^{1/2})\\
 (u_1,u_2) &\longmapsto m_*(u_1\otimes u_2\vert_{G^{(2)}})
\end{align} 
is separately sequentially continuous and coincides with the convolution product on $C^\infty_c(G,\Omega^{1/2})$. 
\end{thm}
\begin{proof}
Under the assumption made on $W_1,W_2$, we can apply \cite[Theorems 8.2.4, 8.2.9]{Horm-1} to find  that the bilinear map
\begin{align}
\cD'_{W_1}(G,\Omega^{1/2})\times \cD'_{W_2}(G,\Omega^{1/2}) &\longrightarrow \cD'_{\rho^*(W_1\bar{\times}W_2)}(G^{(2)},\Omega^{1/2}) \\
 (u_1,u_2) & \longmapsto  u_1\otimes u_2\vert_{G^{(2)}}\nonumber
\end{align}
is well defined, sequentially separately continuous for the natural notion of convergence of sequences in the spaces $\cD'_W$ \cite{Horm-1,GuilleStenb1977}, and also separately continuous for the normal topology of these spaces \cite{BDH}. Above, we have set for convenience \( W_1\bar{\times}W_2 = W_1\times W_2 \cup W_1\times 0 \cup 0\times W_2 \).

To apply $m_*$ and get a continuous map for the same topologies, we restrict ourselves to compactly supported distributions and we get 
\begin{equation}
\cE'_{W_1}(G,\Omega^{1/2})\times \cE'_{W_2}(G,\Omega^{1/2})\overset{(\cdot\otimes\cdot)\vert_{G^{(2)}}}{\longrightarrow}\cE'_{\rho^*(W_1\bar{\times}W_2)}(G^{(2)},\Omega^{1/2})\overset{m_*}{\longrightarrow}\cE'_{W_1\overline{*}W_2}(G,\Omega^{1/2}).
\end{equation}
Indeed, the formulas \eqref{eq:rough-formula-WF-convolution} and \eqref{eq:rough-formula-WF-convolution-2} are still valid here and give the last distribution space above. 
\end{proof}
If $u_1$ or $u_2$ is smooth then $\WF{u_1} \times \WF{u_2}$ is empty and \eqref{eq:our-cond-conv-distribution} is trivially satisfied, thus
\begin{cor}
 The convolution product of Theorem \ref{thm:convol-distrib-on-gpd} gives by restriction the maps 
\begin{equation}
 \cE'(G,\Omega^{1/2})\times C^\infty_c(G,\Omega^{1/2})\overset{*}{\longrightarrow} \cE'_{ s^{-1}_\Gamma(0)}(G,\Omega^{1/2}),
\end{equation}
\begin{equation}
   C^\infty_c(G,\Omega^{1/2})\times \cE'(G,\Omega^{1/2}) \overset{*}{\longrightarrow} \cE'_{r^{-1}_\Gamma(0)}(G,\Omega^{1/2}).
\end{equation}
\end{cor}
As we said, bi-transversal subsets of $T^*G\setminus 0$ satisfy \eqref{eq:our-cond-conv-distribution}. Actually,
\begin{cor}\label{cor:heredity-admissibility}
 Let  $W_1,W_2$ be any subsets of $T^*G\setminus 0$. If $W_1$ is $s$-transversal (resp. $W_2$ is $r$-transversal) then the assumption 
  \eqref{eq:our-cond-conv-distribution} is satisfied and  $W_1* W_2$  is $s$-transversal (resp. $W_2$ $r$-transversal) .
\end{cor}
\begin{proof}
Use the equalities $s_\Gamma\circ m_\Gamma =   s_\Gamma\circ \pr{2}$ and $r_\Gamma\circ m_\Gamma =   r_\Gamma\circ \pr{1}$.
\end{proof}
\begin{rmk}
\rm{Theorems \ref{thm:convolution-transversal-distribution-groupoid}  and  \ref{thm:convol-distrib-on-gpd} do not apply exactly to the same situations. For instance, consider the pair groupoid $G=\mathbb{R}\times \mathbb{R}$. On one hand, using the relation $\ker m_\Gamma=((\ker ds)^\perp\times(\ker dr)^\perp)\cap (T^*G)^{(2)}$ and Remark  \ref{rmk:transv-subm-not-imply-transv-wf}, it is easy to obtain pairs of distributions $(u_1,u_2)\in \cE'_s(\RR^2)\times \cE'(\RR^2)$ for which only Theorem \ref{thm:convolution-transversal-distribution-groupoid} can be applied to define $u_1*u_2$. On the other hand, consider the distributions $u_1=\delta_{(0,0)}$ and $u_2=\delta_{(1,1)}$, whose wave fronts are respectively  $W_1=\{(0,0, \xi, \eta)\ ;\ (\xi, \eta) \neq (0,0)\}$ and $W_2=\{(1,1, \xi, \eta)\ ; \ (\xi, \eta) \neq (0,0)\}$. These distributions are neither $s$ nor $r$ transversal, but $W_1\times W_2 \cap \Gamma^{(2)} =\emptyset$, hence the convolution $u_1 *u_2$ on $G$ can only be defined by Theorem \ref{thm:convol-distrib-on-gpd} (note that $u_1*u_2=0$; less peculiar examples can be easily constructed). 

Of course, both convolution products coincide when both make sense, since  the equality \eqref{eq:bridge-transversal-functorial-WF}  is valid  as soon as   $(\WF{u_1}\times \WF{u_2}) \cap \ker m_\Gamma = \emptyset$.}
\end{rmk}

\section{Appendix : The cotangent groupoid of Coste-Dazord-Weinstein} \label{sec:7}

We recall the definition of the cotangent groupoid of Coste-Dazord-Weinstein. We explain the construction of the source and target map given in \cite{CDW} and we enlighten the role played by the differential of the multiplication map of $G$. This is a pedestrian approach based on concrete differential geometry while more conceptual developments can be found in \cite{Pradines1988,Mackenzie2005}. 

Let $G$ be a Lie groupoid whose multiplication is denoted by $m$, source and target by $s,r$ and inversion by $i$. Differentiating all the structure maps of $G$, we get that $TG\rightrightarrows T G^{(0)}$ is a Lie groupoid whose multiplication is given by $dm$, source and target by $ds,dr$ and inversion by $di$. Hence, it is natural to try to transpose everything to get a groupoid structure on $\Gamma= T^*G$. Following this idea, it is natural to decide that the product $(\gamma_1,\xi_1).(\gamma_2,\xi_2)\in T^*G$ of two elements $(\gamma_j,\xi_j)\in T^*G$ is defined by $(\gamma_1\gamma_2,\xi)$ where $\xi$ is the solution of the equation 
\begin{equation}\label{eq:charac-prod-of-cot-gpd}
  {}^tdm_{(\gamma_1,\gamma_2)}(\xi)=(\xi_1,\xi_2)\vert_{T_{(\gamma_1,\gamma_2)}G^{(2)}}.
\end{equation}
Indeed, $m : G^{(2)}\longrightarrow G$ being a submersion,   ${}^tdm_{(\gamma_1,\gamma_2)}$ is injective for all $(\gamma_1,\gamma_2)\in G^{(2)}$ and  $\xi$, when it exists, is therefore unique. In that case, we have 
\begin{equation}
 \xi = {}^tdm_{(\gamma_1,\gamma_2)}^{-1}\rho(\xi_1,\xi_2)
\end{equation}
where $\rho : T_{G^{(2)}}^*G^{2}\longrightarrow T^* G^{(2)}$ is the restriction of linear forms and we introduce the notations
\begin{equation}\label{eq:def_+_cercle}
 \xi=\xi_1\oplus \xi_2 \text{ and } m_\Gamma(\gamma_1,\xi_1,\gamma_2,\xi_2) = (\gamma_1\gamma_2,\xi_1\oplus \xi_2).
\end{equation}
The equation  \eqref{eq:charac-prod-of-cot-gpd}  has a solution $\xi$ if and only if 
\begin{equation}\label{eq:basic-condit-compo-in-Gamma}
 (\xi_1,\xi_2)\in \im {}^tdm_{(\gamma_1,\gamma_2)}.
\end{equation}
Since $\im {}^tdm_{(\gamma_1,\gamma_2)} = (\ker dm_{(\gamma_1,\gamma_2)})^\perp$,  this is equivalent to
\begin{equation}\label{eq:first-cond-composability-in-Gamma}
 \xi_1(t_1)+\xi_2(t_2)=0,\quad  \forall (t_1,t_2)\in \ker dm_{(\gamma_1,\gamma_2)}.
\end{equation}
Let us explicit $\ker dm \subset TG^{(2)}$. Let 
$$
L_\gamma : G^{s(\gamma)}\longrightarrow G^{r(\gamma)}, \gamma' \mapsto \gamma\gamma' \mbox{  and  } R_\gamma : G_{r(\gamma)}\longrightarrow G_{s(\gamma)}, \gamma' \mapsto \gamma'\gamma
$$
be the left and right multiplication maps of $G$. Let $(\gamma_1,\gamma_2)\in G^{(2)}$ and set $\gamma=\gamma_1\gamma_2$, $x=s(\gamma_1)$.  
Parametrizing 
$G^{(2)}_\gamma=m^{-1}(\gamma)$ by $G^{r(\gamma)}\ni \eta\mapsto (\eta,\eta^{-1}\gamma)$, we find, after a routine computation:
\begin{equation}\label{eq:kernel-dm}
 (t_1,t_2)\in \ker dm_{(\gamma_1,\gamma_2)}\Leftrightarrow t_1=dL_{\gamma_1}di(t),\ t_2=dR_{\gamma_2}(t),\text{ for some } t\in T_xG_x.
\end{equation}
It follows that \eqref{eq:basic-condit-compo-in-Gamma} is equivalent to the equality 
\begin{equation}\label{eq:source-target-Gamma-first-shot}
 {}^tdR_{\gamma_2}(\xi_2)   = -{}^td(L_{\gamma_1}\circ i)(\xi_1) \in (T_xG_x)^*,
\end{equation}
where it is understood that   $R_{\gamma_2}$ and $L_{\gamma_1}\circ i$ are differentiated at  $\gamma=x$ and that the linear forms $\xi_1,\xi_2$ are restricted to the ranges of the corresponding differential maps. The same abuse of notations is used below without further notice. 
We then define  elements $\overline{s}(\xi_1),\overline{r}(\xi_2)$ belonging to $A^*_xG=(T_xG/T_xG^{(0)})^*$ by 
\begin{equation}
 \overline{s}(\xi_1)(t+u)={}^tdL_{\gamma_1}(\xi_1)(t) \text{ for all } t+u \in T_x G^x\oplus T_xG^{(0)}=T_xG,
\end{equation}
\begin{equation}
 \overline{r}(\xi_2)(t+u)={}^tdR_{\gamma_2}(\xi_2) (t) \text{ for all } t+u \in T_x G_x\oplus T_xG^{(0)}=T_xG.
\end{equation}
 Differentiating the relation $\gamma^{-1}\gamma = s(\gamma)$ at $\gamma=x$ we get the relation
\begin{equation}
 di+\id =ds+dr
\end{equation}
which yields 
\(
 -di(t) \equiv t \mod T_xG^{(0)},\ \forall t\in T_xG.
\)
Thus,  \eqref{eq:source-target-Gamma-first-shot}, and then \eqref{eq:basic-condit-compo-in-Gamma}, is equivalent to 
\begin{equation}
 \overline{r}(\xi_2)=\overline{s}(\xi_1) \in A^*_xG.
\end{equation}
This leads to the definitions 
\begin{equation}
 s_\Gamma(\gamma,\xi)= (s(\gamma),\overline{s}(\xi))\in A^*G \text{ and }r_\Gamma(\gamma,\xi)= (r(\gamma),\overline{r}(\xi))\in A^*G,\quad \forall (\gamma,\xi)\in T^*G.
\end{equation}
Finally, we denote $u_\Gamma : A^*G\hookrightarrow T^*G$ the canonical inclusion and we set
\begin{equation}
 i_\Gamma(\gamma,\xi) = (\gamma^{-1},-({}^tdi_{\gamma})^{-1}(\xi)),\quad \forall (\gamma,\xi)\in T^*G.
\end{equation}
\begin{thm}\cite{CDW}.
Let $G$ be a Lie groupoid. The space $\Gamma=T^*G$  is a Lie groupoid with unit space $A^*G$ and structural maps given by  $s_\Gamma,r_\Gamma,m_\Gamma,i_\Gamma$ and $u_\Gamma$ (respectively, source, target, multiplication, inversion and inclusion of unit maps).
\end{thm}
\begin{rmk}
\rm{
\begin{enumerate}
 \item The Lie algebroid of $G$ is sometimes defined by  $A G= \ker ds|_{G^{(0)}}$. In that picture, we deduce from \eqref{eq:source-target-Gamma-first-shot} that $s_\Gamma$ and  $r_\Gamma$ have to be defined by replacing $\overline{s},\overline{r}$ by
\begin{equation}
 \label{eq:alter-defn-s-r-Gamma}
 \widetilde{s}(\xi)=-  {}^td(L_\gamma\circ i)(\xi)  \qquad \text{ and }\qquad  \widetilde{r}(\xi) = {}^tdR_\gamma(\xi).
\end{equation}
\item The submanifold $\Gamma^{(2)}$ of composable pairs in $\Gamma$  is given by 
\begin{equation}
 \Gamma^{(2)} = \{  (\delta_1,\delta_2)\in T_{G^{(2)}}^*G^{2}\ ;\ \rho(\delta_1,\delta_2)\in (\ker dm)^\perp \}
\end{equation}
and $m_\Gamma= {}^tdm^{-1}\circ \rho$. 
\item The graph of $m_\Gamma$ is canonically isomorphic to the conormal space of the graph of $m$:
\begin{equation}
 \mathrm{Gr}(m_\Gamma)\ni (\gamma,\xi,\gamma_1,\xi_1,\gamma_2,\xi_2)\longrightarrow (\gamma,-\xi,\gamma_1,\xi_1,\gamma_2,\xi_2)\in N^*\mathrm{Gr}(m).
\end{equation}
Since $N^*\mathrm{Gr}(m)$ is Lagrangian in $T^*G\times T^*G \times T^*G$, we get that $\mathrm{Gr}(m_\Gamma)$ is Lagrangian in $(-T^*G)\times T^*G \times T^*G$, that is, $\Gamma$ is a symplectic groupoid.
\end{enumerate}
}
\end{rmk}
Finally, we remember that $T^*G$ is also a vector bundle over $G$, and we note $p : T^*G\to G$ the projection map. The following result is  useful and obvious from the construction detailed above.
\begin{prop}
\begin{enumerate}
 \item The subspace of composable pairs $\Gamma^{(2)}$ is a vector bundle over $G^{(2)}$ and $m_\Gamma : \Gamma^{(2)}\to \Gamma$ is a vector bundle homomorphism:
\begin{equation}
 \xymatrix{
    \Gamma^{(2)}\ar[d]^{(p,p)}\ar[r]^{m_\Gamma} & \Gamma \ar[d]^{p}\\
    G^{(2)} \ar[r]^{m} & G
}
\end{equation}
whose kernel is the conormal space of $G^{(2)}$ into $G^2$: 
\(\displaystyle 
 \ker m_\Gamma = N^*G^{(2)}.
\)
\item The maps $r_\Gamma,s_\Gamma : \Gamma \to A^*G$ are also vector bundle homomorphisms:
\begin{equation}
 \xymatrix{
    \Gamma\ar[d]^{p}\ar[r]^{s_\Gamma} &  A^*G \ar[d]^{p}\\
    G  \ar[r]^{s} & G^{(0)}
}\hspace{3cm}
 \xymatrix{
    \Gamma\ar[d]^{p}\ar[r]^{r_\Gamma} &  A^*G \ar[d]^{p}\\
    G  \ar[r]^{r} & G^{(0)}
}
\end{equation}
 and $\displaystyle \ker r_\Gamma = (\ker ds)^\perp$, $\displaystyle \ker s_\Gamma = (\ker dr)^\perp$.
\end{enumerate}
\end{prop}
We finish this review with two basic examples, the first one being the historical one \cite{CDW}.
\begin{exam}
Let $G$ be a Lie group with Lie algebra $\mathfrak g$. We have immediately
\begin{equation}\label{eq:compo-T-star-Lie-group}
  s_\Gamma(g,\xi)= L_{g}^*\xi \in \mathfrak g^* \text{ and } r_\Gamma(g,\xi)=R_{g}^*\xi \in \mathfrak g^*.
\end{equation}
When $s_\Gamma(g_1,\xi_1)=r_\Gamma(g_2,\xi_2) $, we get $({g_1},\xi_1)({g_2},\xi_2)=({g_1}{g_2},\xi)$ with $\xi$ characterized by:
\begin{equation}
  \xi(dm_{(g_1,g_2)}(t_1,t_2))  =  \xi_1(t_1) + \xi_2(t_2).
\end{equation}
Since $dm_{(g_1,g_2)}(t_1,t_2) = dR_{g_2}(t_1)+dL_{g_1}(t_2)$, we obtain $\xi = R_{g_2^{-1}}^*\xi_1=L_{g_1^{-1}}^*\xi_2$. Thus
\begin{equation}
  ({g_1},\xi_1)({g_2},\xi_2) =  (g_1g_2, R_{g_2^{-1}}^*\xi_1)  \text{ when } L_{g_1}^*\xi_1=R_{g_2}^*\xi_2.
\end{equation}

On the other hand, we recall that $G$ acts on $\mathfrak g^*$ by 
\begin{equation}
 \hbox{Ad}^*_{g}.\xi = L_{g}^*R_{{g}^{-1}}^*\xi.
\end{equation}
This gives rise to the transformation groupoid  $ G\semi \mathfrak g^*\rightrightarrows \mathfrak g^*$ whose source, target, multiplication and inversion are thus given by 
\begin{equation}
 s(g,\xi)=\hbox{Ad}^*_{g}.\xi,\quad r(g,\xi)=  \xi ,\quad (g_1,\xi_1)(g_2,\hbox{Ad}^*_{g_1}.\xi_1)=({g_1}{g_2},\xi_1),\quad (g,\xi)^{-1}=(g^{-1},\hbox{Ad}^*_{g}.\xi).
\end{equation}
Now, the vector bundle trivialization 
\(
\Phi:T^*G\longrightarrow G\times\mathfrak g^*,\ 
({g},\xi)\longmapsto ({g},R_{g}^*\xi),
\)
gives a  Lie groupoid isomorphism $\Phi : T^*G \longrightarrow G\semi \mathfrak g^*$. For instance, we check
\begin{eqnarray*}
\Phi(({g_1},\xi_1)({g_2},\xi_2))&=&\Phi(g_1g_2, R_{g_2^{-1}}^*\xi_1) = ( g_1g_2, R_{g_1g_2}^*R_{g_2^{-1}}^*\xi_1)
= ( g_1g_2, R_{g_1}^*\xi_1)\\
&=&({g_1},R_{g_1}^*\xi_1). ({g_2},R_{g_2}^*\xi_2) \quad \text{ since } \hbox{Ad}^*_{g_1}.R_{g_1}^*\xi_1= L_{g_1}^*\xi_1=R_{g_2}^*\xi_2\\ 
&=& \Phi(g_1,\xi_1).\Phi(g_2,\xi_2).
\end{eqnarray*}
\end{exam}

\begin{exam}\label{exam:cot-gpd-of-fiber-product-gpd}
 We take $G=X\times X\times Z\rightrightarrows X\times Z$ (cartesian product of the pair groupoid $X\times X$ with the space $Z$).
Here we have
\[
 \Gamma^{(0)}=A^*G = \{ (x,x,z,\xi,-\xi,0)\ ;\ (x,\xi)\in T^*X, z\in Z\}.
\]
Let $\gamma=(x,y,z)$ and $\xi=(\zeta,\eta,\sigma)\in T^*_\gamma G$. Then $\overline{s}(\xi)\in T^*_{(y,y,z)}X\times X\times Z$ is given by $\eta\in T^*_yX\simeq 0\times T^*_yX\times 0 $ after extension by $0$ onto the subspace of vectors of the form $(u,u,w)$. This is similar for $\overline{r}(\xi)\in T^*_{(x,x,z)}X\times X\times Z $, starting with $\zeta\in T^*_xX\simeq  T^*_x X\times 0 \times 0$.  Using 
\[
 (u,v,w) = (u-v,0,0)+(v,v,w) = (0,v-u,0)+(u,u,w),
\]
we get $s_\Gamma(x,y,z,\xi,\eta,\sigma) =(y,y,z,-\eta,\eta,0)$,  $ r_\Gamma(x,y,z,\xi,\eta,\sigma) =(x,x,z,\xi,-\xi,0)$ and
\begin{equation}
 (x,y,z,\xi,\eta,\sigma). (y,x',z,-\eta,\xi',\sigma') = (x,x',z,\xi,\xi',\sigma+\sigma').
\end{equation}
Note that if $Z=\{\mathrm{pt}\}$, $\Gamma=T^*(X\times X)$ is isomorphic to the pair groupoid $T^*X\times T^*X$, with isomorphism given by 
\(
 T^*(X\times X) \longrightarrow T^*X\times T^*X\ ; \ (x,y,\zeta,\eta)\mapsto (x,\zeta,y,-\eta).
\)
\end{exam}

%
%

%

 \end{document}